\documentclass[11pt,reqno]{amsart}

\usepackage{mysty}
\usepackage{cleveref}
\usepackage[utf8]{inputenc}
\usepackage[OT1]{fontenc}
\usepackage{amsfonts, amsmath, amsthm, amssymb,url, hyperref}
\usepackage{breqn}
\usepackage{graphicx,comment}
\usepackage{verbatim}
\usepackage[margin=1in]{geometry}
\usepackage{xcolor}
\usepackage{cleveref}

\theoremstyle{plain}
\def\Re{\operatorname{Re}}
\def\sym{\operatorname{Sym}}
\def\O{\operatorname{O}}

\def\Im{\operatorname{Im}}
\newcommand{\lf}{\lambda_f}

\theoremstyle{remark}
\newtheorem{remark}[thm]{Remark}

\begin{document}

\title{Explicit zero-free regions for automorphic $L$-functions}

\author{Steven Creech}
\address{Department of Mathematics\\ Brown University\\ Providence, RI 02912\\  United States}
\email{steven_creech@brown.edu}

\author{Alia Hamieh}
\address{Department of Mathematics and Statistics \\
        University of Northern British Columbia \\
        Prince George, BC V2N4Z9 \\
        Canada}
\email{alia.hamieh@unbc.ca}

\author{Simran Khunger}
\address{Department of Mathematics\\University of Michigan\\Ann Arbor, MI 48109\\ United States}
\email{skhunger@umich.edu}

\author{Kaneenika Sinha}
\address{Indian Institute of Science Education and Research Pune\\ Pune  Maharashtra 411008\\ India}
\email{kaneenika@iiserpune.ac.in}

\author{Jakob Streipel}
\address{Department of Mathematics \\
        University at Buffalo \\
        Buffalo, NY 14260 \\
         United States}
\email{jstreipe@buffalo.edu}

\author{Kin Ming Tsang}
\address{Department of Mathematics \\
        University of British Columbia \\
        Vancouver, BC V6T 1Z2 \\
        Canada}
\email{kmtsang@math.ubc.ca}

\thanks{}
\date{\today}

\begin{abstract} 
Let $L(s,f)$ be the $L$-function associated with a newform $f$ of even weight $k$, squarefree level $N$ and trivial nebentypus. In this paper, we establish a new explicit zero-free region for $L(s,f)$. More precisely, we prove that $L(s,f)$ does not vanish in the region $\Re(s)\geq 1-\frac{1}{C\log(kN\max(1,\abs{\Im(s)}))}$ with $C=16.7053$ if $\abs{\Im(s)}\geq 1$ or $\abs{\Im(s)}\leq \frac{0.30992}{\log(kN)}$ and $C=16.9309$ if $\frac{0.30992}{\log(kN)}<\abs{\Im(s)}\leq 1$. This improves a result of Hoey et al.\ where $445.994$ was shown to be an admissible value for $C$.
\end{abstract}



\maketitle

\setcounter{tocdepth}{1}
\tableofcontents

\markboth{\scalebox{0.9}{Creech, Hamieh, Khunger, Sinha, Streipel, Tsang}}{Explicit zero-free regions for automorphic $L$-functions}

\section{Introduction}\label{sec:Introduction} 

Let $k$ and $N$ be positive integers with $k$ even and $N$ squarefree.  Let $f$ denote a Hecke newform of weight $k$ with respect to the congruence subgroup $\Gamma_0(N)=\left\{\begin{bmatrix}
a & b \\
c & d 
\end{bmatrix}\in\mathrm{SL}_2(\mathbb{Z}):\;N\mid c	
\right\}$. It is known that $f$ admits a Fourier expansion of the form
\[
    f(z) = \sum_{n=1}^{\infty}{n^{\frac{k-1}{2}}}\lf(n) q^n, \qquad q = e^{2\pi i z}.
\]
Here,
$\lf(1) = 1$, and for each $n \in \N$, we have
\[
    \frac{T_n(f(z))}{n^{\frac{k-1}{2}}} = \lf(n) f(z),
\]
where $T_n$ denotes the $n$-th Hecke operator acting on the space $S_k(N)$ of cusp forms of weight $k$ with respect to $\Gamma_0(N)$. 

The standard $L$-function $L(s,f)$ corresponding to $f$ is defined as
\[
    L(s,f) = \sum_{n=1}^{\infty} \frac{\lf(n)}{n^s},
\]
which converges absolutely for $\Re(s) > 1$.  $L(s,f)$ has an Euler product representation of the form
\[
    L(s,f) = \prod_p \left(1 - \frac{\lf(p)}{p^s} + \frac{\chi_0(p)}{p^{2s}}\right)^{-1},\qquad\Re(s) > 1,
\]
where $\chi_0(p) = 1$ if $p \nmid N$ and $\chi_0(p) = 0$ if $p \mid N$. It is well-known that $L(s,f)$ can be analytically continued to an entire function on the complex plane, and $L(s,f) \neq 0$ when $\Re(s) > 1$.  This leads to the study of zero-free regions of $L(s,f)$ in the vertical strip $0 \leq \Re(s) \leq 1$. 

In the case of the Riemann zeta function and the Dirichlet $L$-functions, a delineation of explicit zero-free regions combined with explicit zero density estimates leads to vital applications such as effective error terms in the prime number theorem, its analogue for primes in arithmetic progressions and the average orders of several arithmetic functions. The last two decades have witnessed a surge of activity in this direction, most of which aimed at sharpening the results of the 60-year old paper of Rosser and Schoenfeld \cite{rosser-schoenfeld}. Among the many recent contributions to these efforts, we mention \cite{ford}, \cite{kadiri-zeta}, \cite{kadiri-dedekind}, \cite{trudgian},\cite{ramare}, \cite{platt-trudgian},\cite{kadiri-dirichlet}, \cite{KNW}, \cite{broadbent-et-al},   \cite{trudgian-platt2}, \cite{bmor}, \cite{khale}, \cite{bellotti}, and \cite{DGLSW}.

In general, explicit determination of zero-free regions of $L$-functions associated with an arithmetic or algebraic structure helps us better understand the nature and properties of this structure.  For example, the zero-free regions for $L(s,f)$ and the symmetric power $L$-functions $L(s,\sym^m f)$ for $m \geq 2$ (see Section \ref{sec: L-functions of Modular Forms: Definitions, Functional Equations and Explicit Formulae} for the definition) yield effective estimates for the sums 
$\sum_{p \leq x}\lf(p^m)$.
Analogous to the fundamental ideas in the proof of the prime number theorem, wider zero-free regions for $L(s,\sym^m f)$ lead to sharper estimates for these sums.  
These estimates further lead to effective error terms in the Sato-Tate distribution law for the Hecke eigenvalues $\{\lf(p)\}_{p \to \infty} \subset [-2,2]$ of a fixed non-CM Hecke newform $f$ (see, for example, \cite{KM}, \cite{RT}, \cite{Thorner} and \cite{HIJS}).

The first zero-free region result for modular forms beyond the classical context of Dirichlet $L$-functions and their extensions to number fields is due to Moreno \cite{Moreno} who established a zero-free region
for the $L$-function associated with the Ramanujan delta function $\Delta(z)=\sum_{n=1}^{\infty}\tau(n)q^n\in S_{12}(1)$. This zero-free region took the form \begin{equation}\label{eqn:moreno-zfr}
    \Re(s) \geq 1 - \frac{1}{c\log(\abs{\Im(s)} + 2)},
\end{equation}
for some constant $c >0$.  Using \eqref{eqn:moreno-zfr}, Moreno derived the estimate
\[
    \sum_{p \leq x}\tau(p)p^{-\frac{11}{2}} \log p = \O\left(x e^{-A(\log x)^{1/2}}\right),
\]
where $A$ is a suitably chosen positive constant, thereby improving Rankin's estimates in \cite{Rankin1-2} and  \cite{Rankin3}). 

 A generalization of Moreno's zero-free region to all newforms of level $N$ and weight $k\geq1$ is given in \cite[Theorem 5.39]{IK} as follows: there exists an absolute constant $c>0$ such that $L(s,f)$ does not have any zeros in the region
\[
    \Re(s) \geq 1 - \frac{1}{c\log (N(\abs{\Im(s)} + k + 3))},
\]
except possibly one simple real zero, called the exceptional zero, in which case $f$ is self-dual. We know now that this exceptional zero can be provably eliminated thanks to the work of Hoffstein and Ramakrishnan \cite{hoffstein-ramakrishnan} on the non-existence of Siegel zeros for $L$-functions of cusp forms on $\mathrm{GL}(2)$.

For a newform $f$ of squarefree level $N$ and weight $k$, Thorner \cite{Thorner} derived an unconditional estimate for $\sum_{p \leq x}\lf(p^m) \log p$.  An essential ingredient in Thorner's work is the formulation of zero-free regions for $L(s,\sym^m f)$.  Thorner shows that there exists an absolute constant $c>0$ such that $L(s,\sym^m f)$ does not have any zeros in the region 
\[
    \Re(s) \geq 1 - \frac{1}{c\log (\mathfrak q(\sym^m f)(3 + \abs{\Im(s)}))}, 
\]
where $\mathfrak q(\sym^m f)$ denotes the analytic conductor of $L(s,\sym^m f)$.

The ``state-of-the-art'' result in this direction, which also provides explicit constants in the determination of zero-free regions is the following:
\begin{unmthm}[\cite{HIJS}]
Let $N$ be a squarefree positive integer, and let $k$ be an even positive integer.  For a newform $f$ of weight $k$ with respect to $\Gamma_0(N)$, let $\beta_m + it$ be a zero of $L(s,\sym^m f)$ such that $\beta_m > 1/2$.  Then,
\[
    \beta_m \leq  1 - \frac{1}{C_m \log (1.214(m+7) N(k-1)\sqrt{1 + t^2})}, 
\]
where 
\[
    C_m = \frac{(m+7)^2}{2(2 - \sqrt{3})^2}.
\]
In particular, if $m=1$, let $\beta + it$ be a zero of $L(s,f)$ such that $\beta > 1/2$.  Then,
\[
    \beta \leq  1 - \frac{1}{C \log (9.712 N(k-1)\sqrt{1 + t^2})}, 
\]
where \begin{equation}\label{eqn:HIJTconstant}C = \frac{64}{2(2 - \sqrt{3})^2} < 445.994.\end{equation}
\end{unmthm}
The general analytic framework applied to obtain zero-free regions for $L(s,\sym^m f)$  issues from the property that for each $m \geq 1$, $L(s, \sym^mf)$ can be analytically continued to the entire complex plane and has a functional equation.  This follows from the fact that $L(s, \sym^mf)$ is the $L$-function of a cuspidal automorphic representation on $\textrm{GL}_{m+1}(\mathbb A_{\Q})$, proved by Newton and Thorne \cite{NT1,NT2}.\\

We note here that \cite{HIJS} provides explicit constants in the zero-free region of $L(s, \sym^mf)$ for any positive integer $N$. Let $q(\sym^nf)$ denote the arithmetic conductor of $L(s,\sym^n f)$, and let $Q$ be a positive constant such that $q(\sym^nf) \leq Q^{n+1}$ for all $1 \leq n \leq 2m + 1$.  Then, $L(s,\sym^m f) \neq 0$ in the region 
\begin{equation*}\label{m-ZFR}
\Re(s) \geq 1 - \frac{2(2 - \sqrt{3})^2}{(m+7)^2 \log\left (\frac{2}{\sqrt{e}} (k-1)Q(m+7)\sqrt{1 + \abs{\Im(s)}^2}\right)}.
\end{equation*}
If $N$ is squarefree, then $q(\sym^mf) = N^m$ for each $m \geq 1$.  However, such closed form expressions are not available for non-squarefree values of $N$ (see \cite[Remark 2.1]{HIJS}).

The goal of the current article is to further improve the constant $C$ given in \eqref{eqn:HIJTconstant} in the zero-free region for $L(s,f)$.  We prove the following theorem:
\begin{thm}\label{main-thm}
Let $N$ be a squarefree positive integer, and let $k$ be an even positive integer.  For a newform $f$ of weight $k$ with respect to $\Gamma_0(N)$, let $\beta + it$ be a zero of $L(s,f)$ such that $\beta > 1/2$. 
Then, 
\[
\beta \leq \begin{cases}
    1 - \frac{1}{16.7053 \log (kN\abs{t})}&\text{ if }\abs{t} \geq 1\\
    1 - \frac{1}{16.9309 \log (kN)} &\text{ if }\frac{0.30992}{\log (kN)} < \abs{t} < 1\\
   1 - \frac{1}{16.7053 \log (kN)}&\text{ if }\abs{t} \leq \frac{0.30992}{\log (kN)}.
\end{cases}
\]
\end{thm}

In order to prove Theorem \ref{main-thm}, we apply a classical technique due to Ste\v{c}kin \cite{Stechkin} as employed by McCurley \cite{McCurley} in order to obtain explicit zero-free regions for Dirichlet $L$-functions. To the best of our knowledge, this is the first instance in which this technique, in combination with other methods from explicit number theory, has been applied to the setting of $\mathrm{GL}(2)$ automorphic $L$-functions. However, it is worth mentioning here that the methods used in \cite{McCurley} have been further refined, and its results have been improved by Kadiri in \cite{kadiri-dirichlet} in the setting of Dirichlet $L$-functions.

\subsection{Organization of the paper}
 The technique of Ste\v{c}kin, and the preliminaries required for its application to $L(s,f)$ are described in Section \ref{sec:Preparation}.  This includes a discussion of the symmetric power $L$-functions associated to $f$, their functional equations and explicit formulae.  We also describe certain auxiliary functions that will be required in the proof of Theorem \ref{main-thm}.  In Section \ref{sec: Notation and Preliminary Lemmas}, we introduce some notation to be used throughout the article, and delineate the lemmas at the heart of Ste\v{c}kin's differencing argument, along with their modification by McCurley, and further adaptation for the purposes of proving Theorem \ref{main-thm}.  In Section \ref{sec:positivity}, we define a quartic trigonometric polynomial, and indicate how the nonnegativity of this polynomial leads to an inequality (see Proposition \ref{prop:positivity}) involving the Ste\v{c}kin differences of the logarithmic derivatives of the auxiliary $L$-functions introduced in Section \ref{sec: log-der-lf}.  We indicate here the choice of the coefficients of this polynomial which optimizes the upper bounds for the real part of the zeroes of $L(s,f)$ derived from the methods of this article.  In Section \ref{sec: Bounds for Auxiliary Terms}, we derive upper bounds for the Ste\v{c}kin difference of the logarithmic derivatives of each auxiliary $L$-function that appears in Proposition \ref{prop:positivity}.  Finally, in Sections \ref{sec: ZFR for t large} and \ref{sec: ZFR for t small}, we combine the bounds of Section \ref{sec: Bounds for Auxiliary Terms} with the nonnegativity of our quartic trigonometric polynomial to obtain Theorem \ref{main-thm}.  While Section \ref{sec: ZFR for t large} focuses on the case when $\abs{t} \geq 1$, Section \ref{sec: ZFR for t small} covers the case when $\abs{t} < 1$.  We note here that there are subtle but important differences in the treatment of the case when $\abs{t}$ is ``very'' small, that is, when $\abs{t}\leq \frac{\gamma}{\log kN}$ for a carefully chosen $\gamma>0$, and when $\abs{t}$ is not too small, that is, when $\frac{\gamma}{\log kN}<\abs{t}< 1$.  This leads to a slightly different constant in the bound for $\beta$ when $\abs{t}$ is not too small.

\section{Preparation for the Proof of Theorem \ref{main-thm}}\label{sec:Preparation}

\subsection{Isolating the zeros of \texorpdfstring{$L$}{L}-functions}\label{sec:McCurley}

One of the earliest arguments to derive zero-free regions for the Riemann zeta function goes back to the classical proof of the fact that $\zeta(s) \neq 0$ if $\Re(s) = 1$.  This proof depends on the trigonometric inequality
\begin{equation}\label{eqn:trig}3 + 4\cos \theta + \cos 2\theta = 2(1 + \cos\theta)^2 \geq 0, \qquad \forall \theta \in \mathbb{R}.\end{equation}
For $s = \sigma + it,\,\sigma > 1$, we have
\[
    -\frac{\zeta'}{\zeta}(s) = \sum_{n=1}^{\infty} \frac{\Lambda(n)}{n^{\sigma}}\left( \cos(t\log n) - i \sin(t\log n)\right),
\]
where $\Lambda(n)$ is the von Mangoldt function.

By \eqref{eqn:trig}, we have
\begin{equation*}
    \begin{split}
&\Re\left(-3\frac{\zeta'} {\zeta}(\sigma) - 4 \frac{\zeta'}{\zeta}(\sigma + it) - \frac{\zeta'}{\zeta}(\sigma + 2it)\right)
= \sum_{n=1}^{\infty} \frac{\Lambda(n)}{n^{\sigma}} \left(3 + 4\cos (t\log n) + \cos (2t\log n)\right) \geq 0.
\end{split}
\end{equation*}

The above trigonometric inequality motivates several techniques to identify a non-trivial zero of $\zeta(s)$ and obtain wider zero-free regions in the critical strip.  The classical template, due to Hadamard and de la Vallee Poussin, uses the Hadamard factorization for 
$\xi(s) = (s-1)\zeta(s)\Gamma\left(1 + \frac{s}{2}\right)\pi^{-s/2}$ to obtain (see, for example, \cite[Section 10.2]{MV})
\begin{equation}\label{classical-Hadamard}
\Re\left(-\frac{\zeta'}{\zeta}(s)\right) = -\frac{\log \pi}{2} + \Re\left(\frac{1}{s-1}\right) + \frac{1}{2}\Re\frac{\Gamma'}{\Gamma}\left(1 + \frac{s}{2}\right) - \sum_{\rho} \Re\left(\frac{1}{s- \rho} \right),
\end{equation}
where $\sum_{\rho}$ runs over all the zeros of $\zeta(s)$ in the region $0 \leq \Re(s) \leq 1$. Observing that 
\[
    \Re\left(\frac{1}{s- \rho} \right)> 0,
\]
we isolate a single zero $\rho = \beta + it_0$ with $\abs{t_0}\geq 2$ in the above sum, and derive the inequality,
\[
    \Re\left(-\frac{\zeta'}{\zeta}(s)\right) <-\frac{\log \pi}{2} + \Re\left(\frac{1}{s-1}\right) + \frac{1}{2}\Re\frac{\Gamma'}{\Gamma}\left(1 + \frac{s}{2}\right) - \Re\left(\frac{1}{s- (\beta + it_0)}\right).
\]
Choosing $s = \sigma + it_0$ and applying standard bounds to the term $\Re\frac{\Gamma'}{\Gamma}\left(1 + \frac{s}{2}\right)$, we obtain
\[
    -\Re\left(\frac{\zeta'}{\zeta}(\sigma + it_0)\right) < A (\log \abs{t_0}) - \frac{1}{\sigma - \beta},
\]
for a positive constant $A$.
Inserting the above inequality into 
\[
    0 \leq \Re\left(-3\frac{\zeta'} {\zeta}(\sigma) - 4 \frac{\zeta'}{\zeta}(\sigma + it_0) - \frac{\zeta'}{\zeta}(\sigma + 2it_0)\right),
\]
we get
\[
    \frac{4}{\sigma - \beta} < \frac{3}{\sigma - 1} + B\log\abs{t_0},\qquad \text{for }\sigma >1,\,\abs{t_0} \,\geq 2.
\]
Choosing
\[
    \sigma = 1 + \frac{1}{(B+1) \log \abs{t_0}},
\]
one can show that there exists a positive constant $c$ such that if $\zeta(\beta + it_0) = 0$ and $\abs{t_0} \geq 2$, then
\[
    \beta \leq  1 - \frac{1}{c\log \abs{t_0}}.
\]
Using a clever modification of the above argument, Ste\v{c}kin \cite{Stechkin} obtains a precise value for $c$  if $\abs{t_0}$ is sufficiently large.  Ste\v{c}kin's argument is a fundamental ingredient in this article.  We explain it briefly and refer the reader to Section \ref{sec: Notation and Preliminary Lemmas} for further details.
Consider a trigonometric polynomial 
\[
    P_n(\theta) = \sum_{k=0}^n a_k \cos k\theta
\]
such that $a_k \geq 0$ for each $0 \leq k \leq n$, $a_0 < a_1$ and $P_n(\theta) \geq 0$ for all $\theta \in \R$.  

We set $\phi =\frac{1+\sqrt{5}}{2}$ and choose a real number $1<\sigma<\phi$. For a given $\sigma$, we set $\sigma_1 = \frac12+\frac12\sqrt{1+4\sigma^2}$.  We denote $s = \sigma +it$ and $s_1 = \sigma_1 + it$.  Further, we denote
\[
    G(z) = \Re \left(-\frac{\zeta'}{\zeta}(z)\right),\qquad \Re(z) > 1.
\]
Since $P_n(\theta) \geq 0$, we have
\begin{equation}\label{P_n}
    \begin{split}
    &S(t) = \sum_{k=0}^n a_k \left(G(\sigma +ikt) - \frac{1}{\sqrt{5}}G(\sigma_1 + ikt)\right)\\
    &=\sum_{n=2}^{\infty} \Lambda(n) \left(\frac{1}{n^{\sigma}} - \frac{1}{\sqrt{5} n^{\sigma_1}}\right)P_n(t\log n) \geq 0.
    \end{split}
\end{equation}
On the other hand, by an application of \eqref{classical-Hadamard}, Ste\v{c}kin observes that if $\beta + it_0 $ is a zero of the zeta function such  that $1/2 < \beta <1$ and $t_0 \geq T \geq 12$, then, for $k \geq 1$,
\begin{equation}\label{kgeq1}
G(\sigma +ikt_0) - \frac{1}{\sqrt{5}}G(\sigma_1 + ikt_0) \leq \frac{1 - \frac{1}{\sqrt{5}}}{2} \log kt_0 - \frac{1}{\sigma - \beta} + A_3(T),
\end{equation}
where 
\[
    A_3(T) = \left(1 - \frac{1}{\sqrt{5}}\right)\left(\frac{\gamma}{2} + 1 - \frac{3}{2}\log 2 - \log \pi\right) + \left(\frac{\pi}{4} + \frac{4}{T}\right)\frac{1}{T} + \frac{1}{\sqrt{5}}\left(\frac{\pi}{4} + \frac{2}{T}\right)\frac{1}{T}.
\]
We note here that to obtain the above inequality, one requires explicit bounds for terms of the form
\[
    \Re\left(\frac{\Gamma'}{\Gamma}\left(\frac{s}{2} + 1\right) - \frac{1}{\sqrt{5}}\frac{\Gamma'}{\Gamma}\left(\frac{s_1}{2} + 1\right)\right).
\]
Combining \eqref{P_n} and \eqref{kgeq1} with a bound for $G(\sigma) - \frac{1}{\sqrt{5}}G(\sigma_1)$, we isolate $\beta$ as follows:
\[
    \frac{a_1}{\sigma - \beta} < \frac{a_0}{\sigma - 1} + \sum_{k=1}^n a_k \left(\frac{1 - \frac{1}{\sqrt{5}}}{2}\right) \log\, t_0 + \sum_{k=1}^n a_k \left(\frac{1 - \frac{1}{\sqrt{5}}}{2}\right) \log\,k + A_3 \sum_{k=1}^n a_k.
\]
The choice 
\[
    P_4(\theta) =8(0.9126+\cos\theta)^2(0.2766+\cos\theta)^2
\]
yields
\[
    \beta \leq  1 - \frac{1}{9.65\log t_0} .
\]
In other words, $\zeta(s) \neq 0$ in the region
\[
    \sigma > 1 - \frac{1}{9.65 \log \abs{t}} ,\qquad\abs{t}\, \geq 12.
\]
The explicit constant in the zero-free region for $\zeta(s)$ obtained by Ste\v{c}kin's argument is significantly sharper than those obtained before.   McCurley \cite{McCurley} adapted the above argument to derive explicit zero-free regions for Dirichlet $L$-functions, extending Ste\v{c}kin's argument to small values of $\abs{t}$ as well.  For a positive integer $q$, let $\mathcal L_q(s)$ be the product of the $\phi(q)$  Dirichlet $L$-functions $L(s,\chi)$ associated with the characters $\chi$ modulo $q$.  Then, $\mathcal L_q(s)$ has at most a single zero in the region
\begin{equation}\label{eqn:mccurley}\sigma > 1 - \frac{1}{9.65\log M} ,\qquad M = \max\{q, q\abs{t}, 10\}.\end{equation}
This exceptional zero is a simple real zero of a Dirichlet $L$-function associated with a real non-principal character modulo $q$.  
We note here that the result of McCurley has been significantly improved by Kadiri in \cite{kadiri-dirichlet} where the explicit constant $9.65$ in \eqref{eqn:mccurley} has been replaced by $5.60$ (building on the ideas pioneered in \cite{kadiri-zeta}).

In this article, we adapt the ideas of Ste\v{c}kin and McCurley for the modular $L$-function $L(s,f)$ associated with a Hecke newform $f$ of even weight $k$ and squarefree level $N$.  This argument will require us to consider the symmetric power $L$-functions $L(s,\sym^mf)$ for $ 0 \leq m \leq 4$.  In the next section, we review fundamental analytic properties of these $L$-functions, such as the analogue of the classical Hadamard factorization \eqref{classical-Hadamard} for $L(s,\sym^mf)$, which will be required for the proof of Theorem \ref{main-thm}.

\subsection{Symmetric power \texorpdfstring{$L$}{L}-functions: definitions, functional equation and explicit formulae}\label{sec: L-functions of Modular Forms: Definitions, Functional Equations and Explicit Formulae} 

\subsubsection{Definitions}\label{sec: definitions-Symmetric power}
Let $N$ be a positive squarefree integer and let $k$ be a positive even integer. Consider a Hecke newform $f$ of weight $k$ with respect to $\Gamma_0(N)$. We mention here that  $f$ is necessarily non-CM because it has squarefree level and trivial nebentypus (see the proof of the
Corollary of Theorem A in \cite{ramakrishnan-appendix}).

In this section,  we start by reviewing fundamental analytic properties of the symmetric power $L$-functions $L(s,\sym^m f)$ associated with the newform $f$. We refer the reader to \cite[Sections 5.3 and 5.11]{IK}, \cite[Section 1]{LW}, and \cite[Section 2]{HIJS} for further details. 

Following the notation in the introduction, $f$ has a Fourier expansion of the form
\[
    f(z) = \sum_{n=1}^{\infty}{n^{\frac{k-1}{2}}}\lf(n) q^n,
\]
and the standard $L$-function corresponding to $f$ is defined by the Dirichlet series
\[
    L(s,f) = \sum_{n=1}^{\infty} \frac{\lf(n)}{n^s},\qquad \Re(s) > 1.
\]

Moreover, for $\Re(s)>1$, we have the Euler product representation
\begin{align*}
L(f,s) &= \prod_{p\mid N}\left(1-\lambda_f(p)p^{-s}\right)^{-1}\prod_{p\nmid N}\left(1-\lambda_f(p)p^{-s}+p^{-2s}\right)^{-1}\nonumber\\&= \prod_{p}  \left(1-\frac{\alpha_1(p)}{p^s} \right)^{-1}\left(1-\frac{\alpha_2(p)}{p^s} \right)^{-1},
\end{align*}
where $\lambda_f(p)=\alpha_1(p)+\alpha_2(p)$ and more generally $$\lf(p^r) = \frac{\alpha_1(p)^{r+1} - \alpha_2(p)^{r+1}}{\alpha_1(p) - \alpha_2(p)},$$ for all  $r\geq 1$ and all primes $p$. Moreover, $\alpha_1(p)\alpha_2(p)=1$ if $p\nmid N$ and $\alpha_1(p)\alpha_2(p)=0$ if $p\mid N$. The celebrated work of Deligne (see \cite{Deligne1} and \cite{Deligne2}) asserts that $\abs{\alpha_1(p)} = \abs{\alpha_2(p)} =1$ if $p\nmid N$. If $p\mid N$, we shall assume without loss of generality that $\alpha_2(p)=0$. By the Atkin-Lehner-Li theory (see \cite[Theoem~3]{Li}), we infer that 
$\alpha_1(p) = \epsilon_f(p) p^{-\frac12}$ if $p \mathrel{\Vert} N$. Here $\epsilon_f(p)=\pm1$, and $-\epsilon_f(p)$ is the eigenvalue of $f$ under the action of the Atkin-Lehner involution $W_p$.

For any integer $m \geq 0$, the symmetric power $L$-function $L(s,\sym^m f)$ is given by the Euler product
\begin{equation}\label{EPofL}
L(s,\sym^mf) = \prod_p \prod_{0 \leq j \leq m}\left( 1 - \frac{\alpha_1(p)^{m-j}\alpha_2(p)^j}{p^s}\right)^{-1},
\end{equation}
and the Dirichlet series representation 
\[
    L(s,\sym^mf) = \sum_{n=1}^{\infty} \frac{\lambda_{\sym^m f}(n)}{n^s},\qquad\Re(s) > 1.
\]
Note that $L(s,\sym^0f) = \zeta(s)$ and $L(s,\sym^1f) = L(s,f)$. 

\subsubsection{The functional equation and explicit formulae}\label{sec: fn-eqn-explicit-formula} We set
\begin{equation}\label{Gamma-symf}
\gamma(s,\sym^mf) = 
\begin{cases}
\prod_{l=0}^n \Gamma_{\mathbb{C}}\left(s + \left(l + \frac{1}{2}\right)(k-1)\right)&\text{ if }m = 2n+1\\
\Gamma_{\R}\left(s + \delta_{2 \nmid n}\right)\prod_{l=1}^n \Gamma_{\mathbb{C}}(s + l(k-1))&\text{ if }m = 2n,
\end{cases}
\end{equation}
where $\Gamma_{\mathbb{C}}(s) = 2(2\pi)^{-s}\Gamma(s)$, $\Gamma_{\R}(s) = \pi^{-s/2}\Gamma(s/2)$, and $\delta_{2 \nmid n} = 1$ if 2 does not divide $n$, and 0 otherwise.
The completed $L$-function
\[
    \Lambda(s,\sym^mf) = N^{ms/2}\gamma(s,\sym^mf)L(s,\sym^mf)
\]
is entire, and it satisfies the functional equation
\[
    \Lambda(s,\sym^mf) = \epsilon_{\sym^mf}\Lambda(1-s,\sym^mf),
\]
where
$\epsilon_{\sym^mf} = \pm 1$.
We denote
\[
    - \frac{L'}{L}(s,\sym^mf) = \sum_{n=1}^{\infty} \frac{\Lambda_{\sym^mf}(n)\Lambda(n)}{n^s},\qquad\Re(s) > 1,
\]
and as per \eqref{EPofL},
\begin{equation*}\label{Lambda_f(n)}
\Lambda_{\sym^mf}(n) = 
\begin{cases}
    \sum_{j = 0}^m \alpha_{1}(p)^{(m-j)l}\alpha_2(p)^{jl} &\text{ if }n = p^l,\,l \geq 1\\
    0&\text{ otherwise.}
\end{cases}
\end{equation*}

Using the Hadamard factorization of the completed $L$-function $\Lambda(s,\sym^mf)$, we have the following analogue of \eqref{classical-Hadamard} for $L(s,\sym^mf)$ (see \cite[Theorem 5.6]{IK}):
\begin{equation}\label{Hadamard-Symf}
\Re\left(-\frac{L'}{L}\left(s, \sym^mf\right)\right) = \frac{m}{2}\log N + \Re\left(\frac{\gamma'}{\gamma}\left(s, \sym^mf\right)\right) -   \Re\left(\sum_{\substack{\rho: \\ L(\rho, \sym^mf) = 0 \\ 0 \leq \Re(\rho) \leq 1}}\frac{1}{s- \rho}\right).
\end{equation}

\subsubsection{Logarithmic derivatives of certain products of $L$-functions}\label{sec: log-der-lf}

The argument to obtain zero-free regions for $L(s,f)$  requires us to consider a quartic trigonometric polynomial (see Section \ref{sec:positivity}) of the form:
\begin{align*}P_4(\theta;n)=8(a+\Lambda_f(n)\cos\theta)^2(b+\Lambda_f(n)\cos\theta)^2,\end{align*}
for appropriate constants $a$ and $b$. The following lemma provides a relationship between the terms $\Lambda_f(n)^m$ for $0 \leq m \leq 4$ (which appear as we expand $P_4(\theta;n)$) and the Dirichlet coefficients of the logarithmic derivatives of the functions 
\begin{align}\label{eqn:def-power-l-fns}L_{(0)}(s,f) & = \zeta(s),\nonumber\\
L_{(1)}(s,f) &= L(s,f),\nonumber\\
L_{(2)}(s,f)&= L(s,\chi_0) L(s,\sym^2 f),\\
L_{(3)}(s,f) &= L(s,\sym^3 f)L(s,f)^2\prod_{p\mid N} \left(1 - \frac{\alpha_1(p)}{p^s}\right)^2,\nonumber\\
L_{(4)}(s,f) &= L(s,\chi_0)^2 L(s,\sym^2 f)^3 L(s,\sym^4 f) \prod_{p\mid N} \left(1 - \frac{\alpha_1(p)^2}{p^s}\right)^3\nonumber.\end{align}
\begin{lem}\label{L_q-log-der}
For each $0 \leq m \leq 4$, we have
\begin{equation*}\label{log-der-L_q} - \frac{L'_{(m)}}{L_{(m)}}(s,f) = \sum_{n=1}^{\infty}\frac{\Lambda(n)\Lambda_{f,m}(n)} {n^s},\qquad \Re(s) > 1,
\end{equation*}
where $$ \Lambda_{f,m}(n) = \Lambda_f(n)^m.$$
\end{lem}
\begin{proof}
    The lemma is immediate for $m = 0$ and $m = 1$.  
 Observe that
\[
    \Lambda_f(n) = 
    \begin{cases}
        \alpha_1(p)^l + \alpha_2(p)^l&\text{ if }n = p^l,\,l \geq 1,\\
        0&\text{ otherwise.}
    \end{cases}
\]
Here, we recall that $\alpha_1(p)\alpha_2(p) = 1$ for $(p,N) = 1$, and $\alpha_2(p) = 0$ for $p \mid N$. Also, note that
\[
    -\frac{L'}{L}(s,\chi_0) = \sum_{n=1}^{\infty} \frac{\Lambda(n)\chi_0(n)}{n^s},\qquad \Re(s) > 1.
\]
For $m=2$, we have
\begin{equation*}
\begin{split}
\Lambda_{f,2}(n) 
&=  \begin{cases}
    \alpha_1(p)^{2l} + 2\alpha_1(p)^l\alpha_2(p)^l + \alpha_2(p)^{2l}&\text{ if }n = p^l,\,l \geq 1,\,(p,N) = 1\\
\alpha_1(p)^{2l} &\text{ if }n = p^l,\,l \geq 1,\,(p,N) > 1.\\
\end{cases}\\
\end{split}
\end{equation*}
For $m = 3$, we can write
\begin{equation}\label{RSL_3-modified}
L_{(3)}\left(s,f \right) = L(s,\sym^3 f)\prod_{p \nmid N}\left(1 - \frac{\alpha_1(p)}{p^s}\right)^{-2}\left(1 - \frac{\alpha_2(p)}{p^s}\right)^{-2}.
\end{equation}
For $n = p^l,\,l \geq 1$ and $(p,N) = 1$, we use \eqref{eqn:def-power-l-fns} to deduce that
\begin{equation*}
\begin{split}
 &\Lambda_{f,3}(n) = \Lambda_{\sym^3 f}(n) + 2 \Lambda_f(n) \\
& = \Lambda_{\sym^3}(p^l) + 2\Lambda_f(p^l)\\
&=\alpha_1(p)^{3l} + \alpha_1(p)^{2l}\alpha_2(p)^l + \alpha_1(p)^{l}\alpha_2(p)^{2l} + \alpha_2(p)^{3l} + 2\left(\alpha_1(p)^l + \alpha_2(p)^l\right)\\
&= \alpha_1(p)^{3l} + 3\alpha_1(p)^{2l}\alpha_2(p)^l + 3\alpha_1(p)^{l}\alpha_2(p)^{2l} + \alpha_2(p)^{3l} \\
&= (\alpha_1(p)^l + \alpha_2(p)^l)^3\\
&=\Lambda_f(n)^3.\\
\end{split}
\end{equation*}
On the other hand, for $n = p^l,\,l \geq 1$ and $p \mid N$, we use \eqref{RSL_3-modified}. Thus, we have,
\[
    \Lambda_{f,3}(n) = \Lambda_{\sym^3 f}(n) = \alpha_1(p)^{3l} \\ = \Lambda_f(n)^3.
\]
Here, the equality $\Lambda_{\sym^3 f}(n) = \alpha_1(p)^{3l}$ follows from the fact that $\alpha_2(p) = 0$.
Thus,
\begin{equation*}\label{RSL_3-log-der-coeff}
\Lambda_{f,3}(n) = 
 \Lambda_f(n)^3
\end{equation*} for all $n$.
For $m=4$, we have
\begin{equation}
\label{L4-modified}
L_{(4)}\left(s,f \right)= L(s,\chi_0)^2 L(s,\sym^4 f) \prod_{p \nmid N} \left[\left(1 - \frac{\alpha_1(p)^2}{p^s}\right)\left(1 - \frac{\alpha_1(p)\alpha_2(p)}{p^s}\right)\left(1 - \frac{\alpha_2(p)^2}{p^s}\right)\right]^{-3}.
\end{equation}

Using \eqref{eqn:def-power-l-fns} and \eqref{L4-modified}, we have, for $n = p^l,\,l \geq 1$,
\begin{equation*}
    \begin{split}
        &\Lambda_{f,4}(p^l) = 2\chi_0(p^l) + \Lambda_{\sym^4f}(p^l) + \begin{cases}
        3\Lambda_{\sym^2f}(p^l) &\text{ if }(p,N) = 1\\
            0 &\text{ if }(p,N) > 1.
        \end{cases} \\
    \end{split}
\end{equation*}
If $(p,N) = 1$, then
\begin{equation*}
    \begin{split}
  & \Lambda_{f,4}(p^l) = 2 + 3(\alpha_1(p)^{2l} + \alpha_1(p)^l\alpha_2(p)^l + \alpha_2(p)^{2l})\\
  & + \alpha_1(p)^{4l} + \alpha_1(p)^{3l} \alpha_2(p)^{l} + \alpha_1(p)^{2l} \alpha_2(p)^{2l} + \alpha_1(p)^{l}\alpha_2(p)^{3l} + \alpha_2(p)^{4l}\\
  &= 6 + 4(\alpha_1(p)^{2l} + \alpha_2(p)^{2l}) + \alpha_1(p)^{4l} + \alpha_2(p)^{4l}\\
  &= (\alpha_1(p)^l + \alpha_2(p)^l)^4
\end{split}
\end{equation*}
If $(p,N) > 1$, then 
\begin{equation*}
  \Lambda_{f,4}(p^l) = \Lambda_{\sym^4 f}(p^l) = \alpha_1(p)^{4l}, 
  \end{equation*}
  since $\alpha_2(p) = 0.$
Thus,
\begin{equation*}\label{RSL-4-log-der-coeff}
\Lambda_{f,4}(n) = 
 \Lambda_f(n)^4.
\end{equation*}
\end{proof}

\section{Notation and Preliminary Lemmas}\label{sec: Notation and Preliminary Lemmas}

We fix some notation that will be used throughout the article.  We set $\kappa=1-\frac{1}{\sqrt{5}}$ and $\phi=\frac{1+\sqrt{5}}{2}$. For a given $\sigma>1$, we set $\sigma_1=\frac12+\frac12\sqrt{1+4\sigma^2}$. Throughout this paper, we assume that $\sigma\in(1,\phi)$ except in some explicitly indicated instances. Observe that 
\begin{equation*}
\sigma<\phi<\sigma_1<\frac12+\frac12\sqrt{1+4\phi^2}<2.2.
\end{equation*}
 For $t\in\mathbb{R}$, we shall use $s$ to denote the complex number $\sigma+it$ and $s_1$ to denote the complex number $\sigma_1+it$. We define
\begin{equation*}
F(s,z)=\Re\left(\frac{1}{s-z}+\frac{1}{s-1+\overline{z}}\right).
\end{equation*}
The following lemmas are at the heart of Ste\v{c}kin's trick and will be used extensively in this work. 
\begin{lem}\label{lem:Stechkin}
Let $t\in\mathbb{R}$, $s=\sigma+it$ and $s_1=\sigma_1+it$.
\begin{enumerate}
\item[{(a)}] If $0 <\Re(z) < 1$, we have
\begin{equation*}
F(s,z)-\frac{1}{\sqrt{5}}F(s_1,z)\geq 0.
\end{equation*}
\item[{(b)}] If $\frac12\leq \Re(z)<1$ and $\Im(z)=t$, then
\begin{equation*}
\Re\left(\frac{1}{s-1+\overline{z}}\right)-\frac{1}{\sqrt{5}}F(s_1,z)\geq 0.
\end{equation*}
\end{enumerate}
\end{lem}
\begin{proof}
This is  \cite[Lemma 2]{Stechkin}.
\end{proof}

\begin{lem}\label{lem:eq22mccurley}
Let $\frac{1}{2} \leq \Re(z) \leq 1$ and $\abs{\Im z} < 1$. 
Then
\[
\Re\frac{1}{\sigma-1+\overline{z}} -\frac{1}{\sqrt{5}}F(\sigma_1,z) \geq -0.6.
\]
\end{lem}
\begin{proof}
By observing that $F(\sigma_1, z) < \sqrt{5}$,
it suffices to verify that 
\begin{equation}\label{eqn:upper-bound}
-\Re\left(\frac{1}{\sigma-1+\overline{z}}\right) \leq -0.4.
\end{equation}
To get this bound, we remark that 
\[
-\Re\left(\frac{1}{\sigma-1+\overline{z}}\right)=-\frac{\sigma-1+\Re(z)}{(\sigma-1+\Re(z))^2+\Im(z)^2}<-\frac{\sigma-1+\Re(z)}{(\sigma-1+\Re(z))^2+1}.
\]
Since $\frac12<\sigma-1+\Re(z)<\phi$ and the function
$\frac{x}{x^2+1}$
is concave down in $(\frac12,\phi)$,  we have the lower bound $\frac{\sigma-1+\Re(z)}{(\sigma-1+\Re(z))^2 + 1} \geq 0.4$, and \eqref{eqn:upper-bound} follows.  

\end{proof}

\begin{lem}\label{prop:stechkin with 3t}
    Let $\frac12\leq\Re(z) < 1$ and $\Im(s) = 3 \Im(z)$. 
    Then 
    \[
        \Re\left(\frac{1}{s-1+\overline{z}}\right)-\frac{1}{\sqrt{5}}F(s_1,z) \geq -\frac{\phi}{\sqrt{5}} - \frac{2}{5} \geq -1.12361.
    \]
\end{lem}

\begin{proof}
    Observe that
    \[
       \Re\left(\frac{1}{s-1+\overline{z}}\right) = \frac{\sigma - 1 + \Re(z)}{(\sigma - 1 + \Re(z))^2 + 4 \Im(z)^2}>0.
    \]
    We write
    \[
        \frac{1}{\sqrt{5}}F(s_1,z) = \frac{1}{\sqrt{5}} \frac{\sigma_1 - \Re(z)}{(\sigma_1 - \Re(z))^2 + 4 \Im(z)^2} + \frac{1}{\sqrt{5}} \frac{\sigma_1 - 1 + \Re(z)}{(\sigma_1 - 1 + \Re(z))^2 + 4 \Im(z)^2}.
    \]
    Notice that
    \begin{equation*}
       \frac{1}{\sqrt{5}} \frac{\sigma_1 - \Re(z)}{(\sigma_1 - \Re(z))^2 + 4 \Im(z)^2} \leq  \frac{1}{\sqrt{5}} \frac{\sigma_1 - \Re(z)}{(\sigma_1 - \Re(z))^2} = \frac{1}{\sqrt{5}} \frac{1}{\sigma_1 - \Re(z)}\leq \frac{1}{\sqrt{5}} \frac{1}{\phi - 1} = \frac{\phi}{\sqrt{5}},
    \end{equation*}
    where the last inequality follows since $\sigma_1 \geq \phi$ and $\Re(z) < 1$.
   We also have 
    \[
        \frac{1}{\sqrt{5}} \frac{\sigma_1 - 1 + \Re(z)}{(\sigma_1 - 1 + \Re(z))^2 + 4 \Im(z)^2} \leq \frac{1}{\sqrt{5}} \frac{1}{\sigma_1 - 1 + \Re(z)} \leq \frac{1}{\sqrt{5}} \frac{1}{\phi - 1/2} = \frac{2}{5}. 
    \]
Combining all these bounds yields the desired result. \qedhere
\end{proof}

Other expressions of the form  $\Re\left(L(s)-\frac{1}{\sqrt{5}}L(s_1)\right)$ for various functions $L$  will appear prominently throughout the paper, and we will refer to them as Ste\v{c}kin's differences. In what follows, we establish an explicit upper bound for one such expression, namely the Ste\v{c}kin's difference $\Re\left(\frac{\Gamma'}{\Gamma}(s+a)-\frac{1}{\sqrt{5}}\frac{\Gamma'}{\Gamma}(s_1+a)\right)$ with $a\geq0$. This bound will be used in Section \ref{sec: Bounds for Auxiliary Terms}.   

\begin{lem}\label{lem:genLem1&2MC}
Let $t\in\mathbb{R}$, $m \geq 1$, and $a \geq 0$. For $s = \sigma + i m t$ and $s_1 = \sigma_1 + i m t$, we have
\begin{align*}
    &\Re \Bigl( \frac{\Gamma'}{\Gamma}(s + a) - \frac{1}{\sqrt{5}} \frac{\Gamma'}{\Gamma}(s_1 + a) \Bigr)\\
    &\leq \begin{cases}
        \kappa \log(\phi+a)+\dfrac{1}{\sqrt{5}(\phi + a)} & \text{if } m=0 \\
        \dfrac{\kappa}{2} \log((\phi+a)^{2} + m^{2})- \kappa \dfrac{2\phi + 2 a}{(2\phi + 2a)^2 + 4 m^2} +  \dfrac{1}{2(1+a)} +  \dfrac{1}{2\sqrt{5}(\phi + a)} & \text{if } m \neq 0,\ \abs{t} < 1 \\
        \kappa \log(m\abs{t}) + \dfrac{\kappa}{2} \log \Bigl(1+ \Bigl( \dfrac{\phi+a}{mt} \Bigr)^{2} \Bigr) + \dfrac{1}{2(1+a)} +  \dfrac{1}{2\sqrt{5}(\phi + a)} & \text{if } m \neq 0,\ \abs{t} \geq 1.
    \end{cases}
\end{align*}
\end{lem}

\begin{proof}
For all $z$ with $\Re(z)>0$, we have 
\begin{equation}\label{eqn:digamma-expansion}
    \Re \frac{\Gamma'}{\Gamma}(z) = \log\abs{z} - \Re \frac{1}{2z} + \Re \int_0^\infty \frac{x - \lfloor x \rfloor - 1/2}{(x + z)^2} \, d x.
\end{equation}
We apply \eqref{eqn:digamma-expansion} with $z = s + a$ and  $z = s_1 + a$, seeking an upper bound for $\Re \frac{\Gamma'}{\Gamma}(s+a)$ and a lower bound for $\Re \frac{\Gamma'}{\Gamma}(s_1+a)$. 

Note that since $-1/2 \leq x - \lfloor x \rfloor - 1/2 < 1/2$, we have
\[
    \abs*{\Re \int_0^\infty \frac{x - \lfloor x \rfloor - 1/2}{(x + z)^2} \, d x} < \frac{1}{2} \int_0^\infty \frac{1}{\abs{x + z}^2} \, d x \leq \frac{1}{2} \int_0^\infty \frac{1}{(x + \Re z)^2} \, d x = \frac{1}{2 \Re z} .
\]
Moreover, since $1 \leq \sigma < \phi \leq \sigma_1 < 2.2$, we have
\[
- \frac{2\sigma + 2 a}{(2\sigma + 2a)^2 + 4 m^2 t^2} = -\Re \frac{1}{2(\sigma+a + imt)} \leq -\Re \frac{1}{2(\phi+a + imt)} = - \frac{2\phi + 2 a}{(2\phi + 2a)^2 + 4 m^2 t^2}
\]
and
\[
- \frac{2\sigma_1 + 2 a}{(2\sigma_1 + 2a)^2 + 4 m^2 t^2} = -\Re \frac{1}{2(\sigma_1 +a + imt)} \geq -\Re \frac{1}{2(\phi+a + imt)} = - \frac{2\phi + 2 a}{(2\phi + 2a)^2 + 4 m^2 t^2} .
\]
Applying these estimates, we get 
\begin{align*}
    \Re \frac{\Gamma'}{\Gamma}(s + a) 
    &\leq \log \sqrt{(\sigma + a)^{2} + m^{2}t^{2}} - \frac{2\sigma + 2 a}{(2\sigma + 2a)^2 + 4 m^2 t^2} + \frac{1}{2(\sigma + a)} \\
    &\leq \log \sqrt{(\phi + a)^{2} + m^{2}t^{2}} - \frac{2\phi + 2 a}{(2\phi + 2a)^2 + 4 m^2 t^2} + \frac{1}{2(1+a)}
\end{align*}
and
\begin{align*}
    \Re \frac{\Gamma'}{\Gamma}(s_1 + a) &\geq \log \sqrt{(\sigma_{1} + a)^{2} + m^{2}t^{2}} - \frac{2\sigma_{1} + 2 a}{(2\sigma_{1} + 2a)^2 + 4 m^2 t^2} - \frac{1}{2(\sigma_1 + a)} \\
    &\geq\log \sqrt{(\phi + a)^{2} + m^{2}t^{2}} - \frac{2\phi + 2 a}{(2\phi + 2a)^2 + 4 m^2 t^2} - \frac{1}{2(\phi + a)} .
\end{align*}
For the case $m=0$, we get
\begin{align*}
    \Re \Bigl( \frac{\Gamma'}{\Gamma}(s + a) - \frac{1}{\sqrt{5}} \frac{\Gamma'}{\Gamma}(s_1 + a) \Bigr)
    &\leq \log(\sigma+a) - \frac{1}{2(\sigma+a)} + \frac{1}{2(\sigma+a)} \\
    & \qquad {} - \frac{1}{\sqrt{5}} \left( \log(\sigma_1 +a) - \frac{1}{2(\sigma_1 +a)} - \frac{1}{2(\sigma_1 +a)} \right) \\
    &\leq \log(\phi+a) - \frac{1}{\sqrt{5}} \log(\phi+a) + \frac{1}{\sqrt{5}(\sigma_{1} +a)} \\
    &\leq \kappa \log(\phi+a) + \frac{1}{\sqrt{5}(\phi+a)} .
\end{align*}
For the case $m \neq 0$ and $\abs{t}<1$, we get

\begin{align*}
    \Re \Bigl( \frac{\Gamma'}{\Gamma}(s + a) - \frac{1}{\sqrt{5}} \frac{\Gamma'}{\Gamma}(s_1 + a) \Bigr)
    &\leq \log \sqrt{(\phi + a)^{2} + m^{2}t^{2}} - \frac{2\phi + 2 a}{(2\phi + 2a)^2 + 4 m^2 t^2} + \frac{1}{2(1+a)} \\
    & - \frac{1}{\sqrt{5}} \left( \log \sqrt{(\phi + a)^{2} + m^{2}t^{2}} - \frac{2\phi + 2 a}{(2\phi + 2a)^2 + 4 m^2 t^2} - \frac{1}{2(\phi + a)}  \right) \\
    &= \frac{\kappa}{2} \log((\phi+a)^{2} + m^{2}t^2) - \kappa \frac{2\phi + 2 a}{(2\phi + 2a)^2 + 4 m^2 t^2} \\ &+  \frac{1}{2(1+a)} +  \frac{1}{2\sqrt{5}(\phi + a)} \\
    &\leq \frac{\kappa}{2} \log((\phi+a)^{2} + m^{2}) - \kappa \frac{2\phi + 2 a}{(2\phi + 2a)^2 + 4 m^2}\\ & +  \frac{1}{2(1+a)} +  \frac{1}{2\sqrt{5}(\phi + a)} .
\end{align*}

For the case $m \neq 0$ and $\abs{t} \geq 1$, we have

\begin{align*}
    \Re \Bigl( \frac{\Gamma'}{\Gamma}(s + a) - \frac{1}{\sqrt{5}} \frac{\Gamma'}{\Gamma}(s_1 + a) \Bigr)
    &\leq \log \sqrt{(\phi + a)^{2} + m^{2}t^{2}} - \frac{2\phi + 2 a}{(2\phi + 2a)^2 + 4 m^2 t^2} + \frac{1}{2(1+a)} \\
    & - \frac{1}{\sqrt{5}} \left( \log \sqrt{(\phi + a)^{2} + m^{2}t^{2}} - \frac{2\phi + 2 a}{(2\phi + 2a)^2 + 4 m^2 t^2} - \frac{1}{2(\phi + a)}  \right) \\
    &= \frac{\kappa}{2} \log((\phi+a)^{2} + m^{2}t^2) - \kappa \frac{2\phi + 2 a}{(2\phi + 2a)^2 + 4 m^2 t^2} \\ & +  \frac{1}{2(1+a)} +  \frac{1}{2\sqrt{5}(\phi + a)} \\
    &\leq \kappa \log(m\abs{t}) + \frac{\kappa}{2} \log \Bigl(1+(\frac{\phi+a}{mt})^{2} \Bigr) + \frac{1}{2(1+a)} +  \frac{1}{2\sqrt{5}(\phi + a)} .
\end{align*}

\end{proof}

Since Lemma \ref{lem:genLem1&2MC} is mainly applied for $a=0$ or $a$ being a linear expression in the weight $k$, it is useful to rewrite the bounds in the lemma as follows. We have 

\begin{align*}
    &\Re \Bigl( \frac{\Gamma'}{\Gamma}(s + a) - \frac{1}{\sqrt{5}} \frac{\Gamma'}{\Gamma}(s_1 + a) \Bigr)\\
    &\leq \begin{cases}
        \frac{1}{\sqrt{5}\phi}+\kappa \log(\phi) & \text{if } m=0, \ a = 0 \\
        \kappa \log(k) + \frac{1}{\sqrt{5}(\phi+a)}+ \kappa \log\Bigl( \frac{\phi + a}{k} \Bigr) & \text{if } m=0, \ a \neq 0 \\        
        \kappa \log\abs{t} + \dfrac{1}{2} +  \dfrac{1}{2\sqrt{5}\phi}+ \kappa \log(m) + \frac{\kappa}{2} \log\Bigl( 1 + \Bigl( \frac{\phi}{m t} \Bigr)^2 \Bigr) & \text{if } m \neq 0,\ a = 0, \ \abs{t} \geq 1  \\
        \kappa \log\abs{t} + \kappa \log(k) + \dfrac{1}{2(1+a)}\\
         +  \dfrac{1}{2\sqrt{5}(\phi + a)}+ \kappa \log(m) + \frac{\kappa}{2} \log\Bigl( \frac{1}{k^2} + \Bigl( \frac{\phi + a}{k m t} \Bigr)^2 \Bigr) & \raisebox{3ex}{$\text{if } m \neq 0,\ a \neq 0, \ \abs{t} \geq 1 $}\\ 
        - \kappa \dfrac{\phi}{2\phi^2 + 2 m^2} +  \dfrac{1}{2} +  \dfrac{1}{2\sqrt{5}\phi}+\frac{\kappa}{2} \log(\phi^2 + m^2) & \text{if } m \neq 0,\ a = 0,\ \abs{t} < 1 \\
        \kappa \log(k) - \kappa \dfrac{2\phi + 2 a}{(2\phi + 2a)^2 + 4 m^2} +  \dfrac{1}{2(1+a)}\\
        +  \dfrac{1}{2\sqrt{5}(\phi + a)} + \frac{\kappa}{2} \log \Bigl( \frac{(\phi + a)^2 + m^2}{k^2} \Bigr) & \raisebox{3ex}{$\text{if } m \neq 0,\ a \neq 0,\ \abs{t} < 1.$}
    \end{cases}
\end{align*}

We now reformulate Lemma \ref{lem:genLem1&2MC} as the following corollary, explicitly isolating the main term from the error term in a manner suited to the context of our application.

\begin{cor}\label{cor: Gamma Bound with M and C Functions} If $a=0$ or linear in $k$, then we have the bound 
\[
\Re \Bigl( \frac{\Gamma'}{\Gamma}(s + a) - \frac{1}{\sqrt{5}} \frac{\Gamma'}{\Gamma}(s_1 + a) \Bigr)\leq M(a,m,t)+C(a,m,t),
\]
where 
\[
M(a,m,t)=\begin{cases}
    0 & m=0,\ a=0\\
    \kappa\log(k) & m=0,\ a\neq 0\\
    \kappa\log\abs{t} & m\neq 0,\ a=0,\ \abs{t} \geq 1\\
    \kappa\log\abs{t}+\kappa\log(k) & m\neq 0,\ a\neq 0,\ \abs{t}\geq 1\\
    0 & m\neq 0,\ a=0,\ \abs{t}<1\\
    \kappa\log(k) & m\neq 0,\ a\neq 0,\ \abs{t} <1
\end{cases}
\]
and 

\[
C(a,m,t)=\begin{cases}
    \frac{1}{\sqrt{5}\phi}+\kappa \log(\phi) &  m=0,\ a = 0 \\
         \frac{1}{\sqrt{5}(\phi+a)}+ \kappa \log\Bigl( \frac{\phi + a}{k} \Bigr) &  m=0,\ a \neq 0 \\        
         \dfrac{1}{2} +  \dfrac{1}{2\sqrt{5}\phi}+ \kappa \log(m) + \frac{\kappa}{2} \log\Bigl( 1 + \Bigl( \frac{\phi}{m t} \Bigr)^2 \Bigr)  &  m \neq 0,\   a = 0,\  \abs{t} \geq 1  \\
         \dfrac{1}{2(1+a)} +  \dfrac{1}{2\sqrt{5}(\phi + a)}\\
         + \kappa \log(m) + \frac{\kappa}{2} \log\Bigl( \frac{1}{k^2} + \Bigl( \frac{\phi + a}{k m t} \Bigr)^2 \Bigr) &  \raisebox{3ex}{$m \neq 0,\ a \neq 0,\  \abs{t} \geq 1$} \\ 
        - \kappa \dfrac{\phi}{2\phi^2 + 2 m^2} +  \dfrac{1}{2} +  \dfrac{1}{2\sqrt{5}\phi}+\frac{\kappa}{2} \log(\phi^2 + m^2)  & m \neq 0,\ a = 0,\ \abs{t} < 1 \\
        - \kappa \dfrac{2\phi + 2 a}{(2\phi + 2a)^2 + 4 m^2} +  \dfrac{1}{2(1+a)} \\
        +  \dfrac{1}{2\sqrt{5}(\phi + a)} + \frac{\kappa}{2} \log \Bigl( \frac{(\phi + a)^2 + m^2}{k^2} \Bigr) &  \raisebox{3ex}{$m \neq 0,\  a \neq 0,\ \abs{t} < 1.$}
\end{cases}
\]

\end{cor}

\section{A Positivity Argument via a Quartic Trigonometric Polynomial}\label{sec:positivity}

For $n=0,1,2,3,4$, we define the auxiliary functions $j_{f,n}(\sigma, t)$ as Ste\v{c}kin's differences:

\[j_{f,n}(\sigma, t)=\Re\left(\frac{1}{\sqrt{5}}\frac{L_{(n)}'}{L_{(n)}}(\sigma_1+it, f)-\frac{L_{(n)}'}{L_{(n)}}(\sigma+it,f)\right).\]
A key ingredient in the proof of Theorem \ref{main-thm} is a linear combination of the form \[\sum_{n=0}^{4}\sum_{m=0}^{4}c_{n,m}j_{f,n}(\sigma, mt)\] that is positive for all $\sigma>1$ and $t\in\mathbb{R}$. 
In constructing this combination, we adapt the standard approach employed in \cite{McCurley} by considering a twisted version of a degree 4 polynomial of the form $P(\theta)=\gamma(a+\cos(\theta))^2(b+\cos(\theta))^2$. We can write $P(\theta)=\sum_{m=0}^{4}a_m\cos(m\theta)$, where
\begin{align*}
    a_0 &= \frac{\gamma}{2} \Bigl( \frac{3}{4} + 4 a b + b^2 + a^2(1 + 2 b^2) \Bigr) \\
    a_1 &= \frac{\gamma}{2} (a + b)(3 + 4 a b) \\
    a_2 &= \frac{\gamma}{2} (1 + a^2 + 4 a b + b^2) \\
    a_3 &= \frac{\gamma}{2} (a + b) \\
    a_4 &= \frac{\gamma}{8}.
\end{align*}
\begin{prop}\label{prop:positivity}
For $\sigma>1$ and $t\in\mathbb{R}$, we have
\begin{align*}
0&\leq 8a^2b^2j(\sigma,0)+(a_2-4)j_{f,2}(\sigma,0)+3j_{f,4}(\sigma,0)+(a_1-3a_3)j_{f}(\sigma,t)+3a_3j_{f,3}(\sigma,t)\\
&\hspace{1em}+(a_2-4)j_{f,2}(\sigma, 2t)+4j_{f,4}(\sigma, 2t)
+a_3j_{f,3}(\sigma, 3t)
+j_{f,4}(\sigma, 4t).
\end{align*}
\end{prop}

\begin{proof}
By expanding each $j_{f,n}(\sigma,mt)$ as a Dirichlet series, we get
\begin{align*}
 &8a^2b^2j(\sigma,0)+(a_2-4)j_{f,2}(\sigma,0)+3j_{f,4}(\sigma,0)+(a_1-3a_3)j_{f}(\sigma,t)+3a_3j_{f,3}(\sigma,t)\\
&\hspace{1em}+(a_2-4)j_{f,2}(\sigma, 2t)+4j_{f,4}(\sigma, 2t)
+a_3j_{f,3}(\sigma, 3t)
+j_{f,4}(\sigma, 4t)\\&=   \sum_{n\geq 1}\left(1-\frac{1}{\sqrt{5}n^{\sigma_1-\sigma}}\right)\frac{\Lambda(n)}{n^\sigma}\Big(8a^2b^2+(a_2-4)\Lambda_{f,2}(n)+3\Lambda_{f,4}(n)\\&\hspace{12em}
    + \left((a_1-3a_3)\Lambda_f(n)+3a_3\Lambda_{f,3}(n)\right)\cos(t\log n)
   \\&\hspace{12em} +\left((a_2-4)\Lambda_{f,2}(n)+4\Lambda_{f,4}(n)\right)\cos(2t\log n)
 \\&\hspace{12em}   +\left(a_3\Lambda_{f,3}(n)\right)\cos(3t\log n)
    +\left(\Lambda_{f,4}(n)\right)\cos(4t\log n)\Big).
\end{align*}
Using Lemma \ref{L_q-log-der}, we can write
\begin{align*}
 &8a^2b^2j(\sigma,0)+(a_2-4)j_{f,2}(\sigma,0)+3j_{f,4}(\sigma,0)+(a_1-3a_3)j_{f}(\sigma,t)+3a_3j_{f,3}(\sigma,t)\\
&\hspace{1em}+(a_2-4)j_{f,2}(\sigma, 2t)+4j_{f,4}(\sigma, 2t)
+a_3j_{f,3}(\sigma, 3t)
+j_{f,4}(\sigma, 4t)\\&=  \sum_{n\geq 1}\left(1-\frac{1}{\sqrt{5}n^{\sigma_1-\sigma}}\right)\frac{\Lambda(n)}{n^\sigma}\Big(8a^2b^2+(a_2-4)\Lambda_f(n)^2+3\Lambda_{f}(n)^4\\&\hspace{12em}
    + \left((a_1-3a_3)\Lambda_f(n)+3a_3\Lambda_{f}(n)^3\right)\cos(t\log n)
   \\&\hspace{12em} +\left((a_2-4)\Lambda_{f}(n)^2+4\Lambda_{f}(n)^4\right)\cos(2t\log n)
 \\&\hspace{12em}   +a_3\Lambda_{f}(n)^3\cos(3t\log n)
    +\Lambda_{f}(n)^4\cos(4t\log n)\Big)\\ &=\sum_{n\geq 1}\left(1-\frac{1}{\sqrt{5}n^{\sigma_1-\sigma}}\right)\frac{\Lambda(n)}{n^\sigma}P_4(t\log n;n),
\end{align*}
where $P_4(\theta;n)=8(a+\Lambda_f(n)\cos\theta)^2(b+\Lambda_f(n)\cos\theta)^2$,  which is clearly non-negative.

\end{proof}

If $\beta_0+it$ is a zero of $L(s,f)$ with $\frac12<\beta_0$ and $\abs{t}\geq1$, our computations in Section \ref{sec: ZFR for t large} will reveal that 
\[
    \beta_0 \leq 1 - \left(\frac{A + B - 2\sqrt{AB}}{C\kappa}\right)\frac{1}{\log (Nk\abs{t})},
\]
with $A = a_1 + 3a_3$, $B = 8a^2b^2 + a_2 + 2$ and $C = a_1 + 2a_2 + 13a_3 + 32$. Hence, for an optimal zero-free region using the methods of this paper, we seek a choice of $(a,b)$ that maximizes the quantity $\frac{A + B - 2\sqrt{AB}}{C\kappa}$. This is accomplished by choosing  $a = 1.5315, b = 0.374949$. Notice that this choice of $(a,b)$ differs from the one made in \cite{McCurley}. The choice of $\gamma > 0$ does not affect the optimality of the zero-free region, and so we choose $\gamma=8$ which yields a monic polynomial. It is noted in \cite{McCurley} that their polynomial is nearly optimal for their purposes, and it is the same as the one used by Rosser and Schoenfeld for $\zeta(s)$ in \cite{rosser-schoenfeld-math-comps}. Here, we choose the optimal polynomial for the method used. Henceforth, we set
\begin{align}&\gamma=8,\; a=1.5315,\; b=0.374949,\label{eqn:coefficients-1}\\
&a_0=24.77002742,\;a_1=40.39336536,\label{eqn:coefficients-2}\\&a_2=23.13206631,\;
a_3=7.625796,\; \text{and}\; a_4=1.\label{eqn:coefficients-3} \end{align}

\section{Bounds for Auxiliary Terms}\label{sec: Bounds for Auxiliary Terms}

In each of the following subsections we will establish bounds for the functions $j_{f,n}(\sigma, t)$ for $n = 0, 1, 2, 3$ and $4$.

\subsection{Bounding \texorpdfstring{$j_{f,0}(\sigma,0)$}{jf0(sigma,0)}}\label{subsec: Bounding jf0(sigma,0)}

\begin{lem}\label{lem: j_f0(sigma,0) bound}
Let $1<\sigma<\phi$.  We have
\begin{equation*}\label{eqn:jf0-simplified}
j_{f,0}(\sigma, 0)\leq j_{f,0}^{\text{Main}}(\sigma, 0)+j_{f,0}^{\text{Error}}(\sigma, 0),
\end{equation*}
where
\begin{equation}\label{eq: main term for small j_1 in all cases}
j_{f,0}^{\text{Main}}(\sigma, 0)=\frac{1}{\sigma-1},
\end{equation}
and 
\begin{equation}\label{eq: error term for j_1 in all cases}
j_{f,0}^{\text{Error}}(\sigma, 0)=-0.601655.
\end{equation}
\end{lem}
\begin{proof}
Recall that \[j_{f,0}(\sigma, 0)=\frac{1}{\sqrt{5}}\frac{\zeta'}{\zeta}(\sigma_1)-\frac{\zeta'}{\zeta}(\sigma).\]
By \eqref{classical-Hadamard}, we have
\begin{align*}
    -\logd{\zeta}(\sigma) &= \frac{1}{\sigma-1}-\frac{1}{2} \log\pi + \frac{1}{2} \logd{\Gamma} \left( \frac{\sigma + 2}{2} \right) - \Re \sum_{\rho} \frac{1}{\sigma - \rho} \\
    &\leq \frac{1}{\sigma-1}-\frac{1}{2} \log\pi + \frac{1}{2} \logd{\Gamma} \left( \frac{\sigma + 2}{2} \right),
\end{align*}
where the sum is over all non-trivial zeros $\rho$ of $\zeta$ can be discarded since $\sigma > \Re(\rho)$ for each $\rho$. 
Applying the standard upper bound \begin{align*}\label{eq:digamma-bd-2}
    \logd{\Gamma}(x) \leq 2 - \gamma - 2 \log 2 + \left( \frac{\pi^{2}}{4} - 2 \right) (2x-3),
\end{align*}
we get 
\[
    j_{f,0}(\sigma,0) \leq  \frac1{\sqrt{5}}\logd{\zeta}(\sigma_1)+ \frac1{\sigma-1}-\frac1{2}\log\pi + 1 - \frac{\gamma}{2} - \log 2 + \left(\frac{\pi^2}{8}-1\right)(\sigma-1),
\]
where 
\begin{equation}\label{eq: Bound on logderiv zeta sigma1}
 \frac1{\sqrt{5}}\logd{\zeta}(\sigma_1)=-\frac1{\sqrt{5}}\sum_{p \text{ prime}} \frac{\log{p}}{p^{\sigma_1}-1} < -\frac{1}{\sqrt{5}}\sum_{\substack{p \text{ prime} \\ p \leq 10000}} \frac{\log{p}}{p^{\sigma_1}-1}  < -0.19197 . 
\end{equation}
We conclude the lemma by observing that
\[-\frac1{2}\log\pi + 1 - \frac{\gamma}{2} - \log 2 + \left(\frac{\pi^2}{8}-1\right)(\phi-1)-0.19197 \leq -0.601655.\qedhere\]
\end{proof}

\subsection{Bounding \texorpdfstring{$j_{f,1}(\sigma,t)$}{jf(sigma,t)}}\label{sec:jf}

\begin{lem}\label{lem: j_f(sigma,t) bound}

Let $\beta_0 + it_{0}$  be a zero of $L(s,f)$ such that $\beta_0 > 1/2$. Let $1<\sigma<\phi $. Then if $0\neq t=t_0$, we have,
\begin{equation*}\label{jf-bound-simplified-t}
j_{f,1}(\sigma,t) = j_f(\sigma,t) \leq j_f^{\text{Main}}(\sigma,t) + j_f^{\text{Error}}(\sigma,t),
\end{equation*}
where
\begin{equation*}\label{jf-bound-simplified-t-Main}
j_f^{\text{Main}}(\sigma,t) = \begin{cases} \frac{\kappa}{2}\log N + \kappa \log (\abs{t}) + \kappa \log k - \frac{1}{\sigma - \beta_0}&\text{ if }\abs{t} \geq 1\\
\frac{\kappa}{2}\log N + \kappa \log k - \frac{1}{\sigma - \beta_0}&\text{ if }\abs{t} < 1,\\
\end{cases}
\end{equation*}
and
\begin{equation*}\label{jf-bound-simplified-t-Error}
j_f^{\text{Error}}(\sigma,t) = \begin{cases} 
    {-0.48973} & \text{ if }\abs{t} \geq 1 \\
    {-0.596438} & \text{ if }\abs{t} < 1.
\end{cases}
\end{equation*}
On the other hand, if $t = 0$ and $\abs{t_0} < 1$, we have
\begin{equation*}\label{jf-bound-simplified-t-0}
j_f(\sigma,0) \leq j_f^{\text{Main}}(\sigma,0) + j_f^{\text{Error}}(\sigma,0),
\end{equation*}
where
\begin{equation}\label{jf-bound-simplified-t-0-Main}
j_f^{\text{Main}}(\sigma,0) = 
\frac{\kappa}{2}\log N + \kappa \log k - 2\frac{\sigma - \beta_0}{(\sigma - \beta_0)^2 + t_0^2}
\end{equation}
and
\begin{equation}\label{jf-bound-simplified-t-0-Error}
j_f^{\text{Error}}(\sigma,0) = {0.603562}.
\end{equation}
\end{lem}
\begin{proof}
Using Lemma \ref{lem:genLem1&2MC}, we have
\begin{align}
    \begin{split}\label{eq: Intermediate Bound for jf(sigma,t)}
    j_f(\sigma,t) 
    \leq{}& \frac{\kappa}{2} \log N - \kappa \log(2\pi) + M \Bigl(\frac{k-1}{2},1, t \Bigr) + C \Bigl(\frac{k-1}{2},1, t \Bigr)  \\
    &- \Re \left( \sum_\rho \frac{1}{\sigma+it-\rho} - \frac1{\sqrt{5}} \sum_\rho \frac1{\sigma_1+it-\rho} \right).
    \end{split}
\end{align}
If $\rho$ is a non-trivial zero for $L(s,f)$, then so is $1- \overline{\rho}$. Hence,
\[ 
\sum_{\substack{\rho \\ 0 < \Re(\rho) < 1}} \frac1{\sigma+it-\rho} = \sideset{}{'}\sum_{\substack{\rho \\ \Re(\rho) \geq 1/2}} \left(\frac1{\sigma+it-\rho} + \frac{1}{\sigma + it - 1 +\overline{\rho}}\right),
\]
where we use $\sideset{}{'}\sum$ to indicate that the terms of the sum with $\Re(\rho) = 1/2$ are counted with weight $1/2$ since when $\Re \rho = \frac{1}{2}$, we have
\[
\dfrac{1}{2} \left( \dfrac{1}{\sigma+it-\rho} + \dfrac{1}{\sigma + it - 1 +\overline{\rho}} \right) = \dfrac{1}{\sigma+it-\rho} .
\]
By Lemma \ref{lem:Stechkin} (a), we know that for each $\rho$ such that $0 < \Re(\rho) < 1$, we have 
\[ 
    F(\sigma + it,\rho) - \frac{1}{\sqrt{5}} F(\sigma_1 + it,\rho) \geq 0,
\] 
where
\begin{equation*}
    F(s,\rho) = \Re\Bigl(\frac{1}{\sigma+it-\rho} + \frac{1}{\sigma + it - 1 +\overline{\rho}} \Bigr) .
\end{equation*}
It follows that
\begin{align}\label{eqn:sum-zeros-upper-bound}
    - \Re \left( \sum_\rho \frac{1}{\sigma+it-\rho} - \frac1{\sqrt{5}} \sum_\rho \frac1{\sigma_1+it-\rho} \right) &\leq 
        -\left( F(\sigma + it, \rho_{0}) - \frac{1}{\sqrt{5}} F(\sigma_1 + it,\rho_{0})\right).
\end{align}
Moreover, by applying Lemma \ref{lem:Stechkin} (b),  if $t  = t_{0}$ and $\frac{1}{2} < \beta_{0} < 1$, we get that
\begin{equation}\label{eqn:lower-bound-F-0}
    F(\sigma+it,\rho_0)-\frac{1}{\sqrt{5}}F(\sigma_1+it,\rho_0)\geq \frac{1}{\sigma-\beta_0}.
\end{equation}
Applying the bounds \eqref{eqn:sum-zeros-upper-bound} and \eqref{eqn:lower-bound-F-0} to the right hand side of  \eqref{eq: Intermediate Bound for jf(sigma,t)} yields
\begin{align*} 
    j_f(\sigma, t) \leq 
        \dfrac{\kappa}{2} \log N - \kappa \log(2\pi) + M\left(\dfrac{k-1}{2},1,t\right) + C\left(\dfrac{k-1}{2},1,t\right) -\dfrac{1}{\sigma-\beta_0}.
\end{align*}
Using Corollary \ref{cor: Gamma Bound with M and C Functions} in the case $\abs{t}\geq1$ case, we see that 
\begin{equation*}\label{eq: Main Term for j_f(sigma,t) for t large}
    M\left(\frac{k-1}{2},1,t\right)= \kappa\log\abs{t}+\kappa\log(k)
\end{equation*}
and  
\begin{align*}\label{eq: Error Term for j_f(sigma,t) for t large}
 - \kappa \log(2\pi)+ C\left(\dfrac{k-1}{2},1,t\right)&\leq - \kappa \log(2\pi)+  \dfrac{1}{2(1+\frac{k-1}{2})} +  \dfrac{1}{2\sqrt{5}(\phi + \frac{k-1}{2})} + \frac{\kappa}{2} \log\Bigl( \frac{1}{k^2} + \Bigl( \frac{\phi + \frac{k-1}{2}}{k t} \Bigr)^2 \Bigr) \\&\leq -0.48973.
\end{align*}
Hence,
\[ j_f(\sigma, t) \leq \frac{\kappa}{2}\log N+\kappa\log\abs{t}+\kappa\log(k)-\frac{1}{\sigma-\beta_0}-0.48973 \]
as desired. For $\abs{t}<1$, we apply Corollary \ref{cor: Gamma Bound with M and C Functions} again to get
\begin{equation*}\label{eq: Main Term for j_f(sigma,t) for t small but not too small}
     M\left(\dfrac{k-1}{2},1,t\right)=\kappa\log(k)
\end{equation*}
and 
\begin{align*}
 - \kappa \log(2\pi)+ C\left(\dfrac{k-1}{2},1,t\right)   &\leq - \kappa \log(2\pi)- \kappa \dfrac{2\phi + k-1}{(2\phi + k-1)^2 + 4} +  \dfrac{1}{k+1} +  \dfrac{1}{2\sqrt{5}(\phi + \frac{k-1}{2})} \\&\hspace{2em}+ \frac{\kappa}{2} \log \left(\frac{\left( \phi + \frac{k-1}{2}\right)^2 + 1}{k^2}\right)\\& \leq -0.596438 .
\end{align*}

In order to handle the case where $t=0$ and $\abs{t_0}<1$,
we observe that
\begin{align}\label{eqn:sum-zeros-jf(sigma,0)}
    - \Re \left( \sum_\rho \frac{1}{\sigma-\rho} - \frac1{\sqrt{5}} \sum_\rho \frac1{\sigma_1-\rho} \right)& \leq 
        - \left( F(\sigma , \rho_{0}) - \frac{1}{\sqrt{5}} F(\sigma_1 ,\rho_{0})+F(\sigma,\overline{\rho_0})-\frac{1}{\sqrt{5}}F(\sigma_1,\overline{\rho_0}) \right)\nonumber\\&\leq  -\Re\left(\frac{1}{\sigma-\rho_0}+\frac{1}{\sigma-\overline{\rho_0}}\right)+1.2,
\end{align}
where the last inequality follows from Lemma~\ref{lem:eq22mccurley}.
Setting $t=0$  in \eqref{eq: Intermediate Bound for jf(sigma,t)} and applying estimate \eqref{eqn:sum-zeros-jf(sigma,0)}, we get
\begin{align}
    \begin{split}\label{eq: Intermediate Bound for jf(sigma,0)}
    j_f(\sigma,0) 
    \leq{}& \frac{\kappa}{2} \log N - \kappa \log(2\pi) + M \Bigl(\frac{k-1}{2},1, 0 \Bigr) + C \Bigl(\frac{k-1}{2},1, 0 \Bigr)  \\
    &-2\frac{\sigma-\beta_0}{(\sigma-\beta_0)^2+t_0^2}+1.2.
    \end{split}
\end{align}
The desired upper bound is achieved by applying 
Corollary \ref{cor: Gamma Bound with M and C Functions} to bound $M \Bigl(\frac{k-1}{2},1, 0 \Bigr) + C \Bigl(\frac{k-1}{2},1, 0 \Bigr)$. \qedhere
\end{proof}

\subsection{Bounding \texorpdfstring{$j_{f,2}(\sigma,mt)$}{jf2(sigma,mt)}}

\begin{lem}\label{lem: j_f2(sigma,mt) bound}
Let $m=0$ or $2$, and $1<\sigma<1.15$.  Then we have 

\begin{equation}\label{jf2mt-bound-simplified}
    j_{f,2}(\sigma,mt) \leq j_{f ,2}^{\text{Main}}(\sigma,mt) + j_{f ,2}^{\text{Error}}(\sigma,mt),
    \end{equation}
where
\begin{equation}\label{jf2mt-bound-main}
j_{f,2}^{\text{Main}}(\sigma,mt) = \kappa \log N + \kappa \log k + \begin{cases}\frac{1}{\sigma - 1} &\text{ if }m=0\\
\frac{\sigma - 1}{(\sigma - 1)^2 + 4t^2} &\text{ if }m=2,\,\abs{t} < 1\\
2\kappa \log \abs{t}&\text{ if }m=2,\,\abs{t} \geq 1\\
\end{cases}
\end{equation}
and
\begin{equation}\label{jf2mt-bound-error}
\begin{split}
j_{f,2}^{\text{Error}}(\sigma,mt)
 = \begin{cases}{-1.66308} -s(N) &\text{ if }m=0\\
{-0.700692} + T(N,1) &\text{ if }m=2,\,\abs{t} < 1\\
{-0.292925} + T(N,1) &\text{ if }m=2,\,\abs{t} \geq 1,\\
\end{cases}
\end{split}
\end{equation}
where 
\[
    s(N) = \sum_{p\mid N} \log\,p \left(\frac{1}{p^{\sigma} - 1} - \frac{1}{\sqrt{5}(p^{\sigma_1} - 1)}\right),
\]
and
\[
    T(N,1) = \sum_{p\mid N}\log p\left(\frac{1}{p^{\sigma} - 1} + \frac{1}{\sqrt{5}(p^{\sigma_1} - 1)}\right). 
\]

\end{lem}

\begin{proof}

By \eqref{eqn:def-power-l-fns} we have
\begin{equation}\label{jf2m}
    \begin{split}
        j_{f ,2}(\sigma, mt)={}& \Re \left(\frac{1}{\sqrt{5}}\frac{L'}{L}(\sigma_1 + imt,\chi_0) - \frac{L'}{L}(\sigma + imt,\chi_0)\right)\\
        & + \Re \left(\frac{1}{\sqrt{5}}\frac{L'}{L}(\sigma_1 + imt,\sym^2f) - \frac{L'}{L}(\sigma + imt,\sym^2f)\right).
    \end{split}
\end{equation}
Assume that $1 < \sigma < 1.15$.  Using \cite[Lemma 3 and Lemma 7]{McCurley}, we get
\begin{equation}\label{Dirichlet-principal,m=0,2}
    \begin{split}
  &\Re \left(\frac{1}{\sqrt{5}}\frac{L'}{L}(\sigma_1 + imt,\chi_0) - \frac{L'}{L}(\sigma + imt,\chi_0)\right)   \\
  < &\begin{cases}
  \frac{1}{\sigma - 1} - 0.8973 - s(N)&\text{ if }m=0\\
      \frac{\sigma -1}{(\sigma - 1)^2 + 4t^2} + 0.1565 - \frac{\kappa}{2} \log \pi + T(N,1) &\text{ if } m = 2,\,\abs{t} < 1\\
      \frac{\kappa}{2} \log \abs{t} + 0.3530 - \frac{\kappa}{2} \log \pi + T(N,1) &\text{ if }m = 2,\,\abs{t} \geq 1.
  \end{cases}
    \end{split}
\end{equation}
Recall, by \eqref{Hadamard-Symf},
\begin{equation*}
\Re\left(\frac{L'}{L}\left(s, \sym^2f\right)\right) = -\log N -\Re\left(\frac{\gamma'}{\gamma}\left(s, \sym^2f\right)\right) +  \Re\left(\sum_{\rho}\frac{1}{s- \rho}\right),
\end{equation*}
where the last sum  runs over the zeros $\rho$ of $L(s,\sym^2 f)=0$ with $0<\Re(\rho)<1$. It follows by Lemma \ref{lem:Stechkin} that
\begin{equation*}
    \begin{split}
        &\Re \left(\frac{1}{\sqrt{5}}\frac{L'}{L}(\sigma_1 + imt,\sym^2f) - \frac{L'}{L}(\sigma + imt,\sym^2f)\right)\\
        &< \kappa \log N + \Re\left(\frac{\gamma'}{\gamma}\left(\sigma + imt , \sym^2f\right)\right) - \frac{1}{\sqrt{5}}\Re\left(\frac{\gamma'}{\gamma}\left(\sigma_1 + imt , \sym^2f\right)\right).\\
        \end{split}
        \end{equation*}
Applying \eqref{Gamma-symf}, we can write
\[
    \frac{\gamma'}{\gamma}\left(s , \sym^2f\right) = - \frac{\log \pi}{2} - \log (2\pi) + \frac12\frac{\Gamma'}{\Gamma}\left(\frac{s+1}{2}\right) + \frac{\Gamma'}{\Gamma}\left(s + k-1\right).
\]
By \cite[Lemma 1 and Lemma 2]{McCurley}, we have
\begin{equation}\label{digamma-1-m=0,2}
    \begin{split}
        &\frac{1}{2} \Re\left(\frac{\Gamma'}{\Gamma}\left(\frac{\sigma + imt + 1}{2}\right) - \frac{1}{\sqrt{5}}\frac{\Gamma'}{\Gamma}\left(\frac{\sigma_1 + imt + 1}{2}\right)\right)\\
        < & 
        \begin{cases}
            0.2469 &\text{ if }m=0,\,\\
            0.2469 &\text{ if }m=2,\,\abs{t} < 1\\
            0.3915 + \frac{\kappa}{2} \log \abs{t} &\text{ if }m = 2,\,\abs{t} \geq 1.
        \end{cases}
        \end{split}
\end{equation}
By Corollary \ref{cor: Gamma Bound with M and C Functions}, we also have
\begin{equation}\label{digamma-2-m=0,2}
        \frac{\Gamma'}{\Gamma}(\sigma + imt + k-1) - \frac{1}{\sqrt{5}}\frac{\Gamma'}{\Gamma}(\sigma_1 + imt + k-1)\\
        < M(k-1,m,t) + C(k-1,m,t).
\end{equation}
Thus, combining equations \eqref{jf2m}, \eqref{Dirichlet-principal,m=0,2}, \eqref{digamma-1-m=0,2} and \eqref{digamma-2-m=0,2}, we conclude that if $1 < \sigma < 1.15$, then  

\begin{equation}\label{jf2m-bound}
     j_{f,2}(\sigma,mt) < \kappa \log N - \frac{\kappa}{2} \log \pi - \kappa \log (2\pi) +  M(k-1,m,t) + C(k-1,m,t)+\ell(\sigma,m,t),
 \end{equation}
 where
 \begin{equation}\label{lmt}
 \ell(\sigma,m,t)=\begin{cases}
      \frac{1}{\sigma-1} - s(N) -0.6504 &\text{ if }m=0\\
        \frac{\sigma-1}{{(\sigma - 1)^2 + 4t^2}} + 0.0870044 + T(N,1)&\text{ if }m=2,\,\abs{t} < 1\\
        \kappa \log \abs{t} + 0.428104 + T(N,1) &\text{ if }m=2,\,\abs{t} \geq 1.\\
     \end{cases}
 \end{equation}

 The proof of the lemma concludes by applying the bounds for $M(k-1,m,t)$ and $C(k-1,m,t)$  from Corollary \ref{cor: Gamma Bound with M and C Functions} in each of the cases to get the desired values for $j_{f ,2}^{\text{Main}}(\sigma,mt)$ and $j_{f,2}^{\text{Error}}(\sigma,mt)$.

\end{proof}

\begin{remark} We will need an upper bound for $j_{f,2}(\sigma,0)$ when allowing $\sigma$ to go all the way up to $\phi$. In this case, it suffices to observe that  if $1<\sigma<\phi$, then
\begin{equation}\label{eqn:jf2mt-bound-error-sigma-phi}
    j_{f,2}^{\text{Error}}(\sigma,0) = -1.36737- s(N).
\end{equation}
\end{remark}
\subsection{Bounding \texorpdfstring{$j_{f,3}(\sigma,t)$}{jf3(sigma,t)}}

\begin{lem}\label{lem: j_f3(sigma,t) bound}
    Let $\beta_0 + it_{0}$  be a zero of $L(s,f)$ such that $\beta_0 > 1/2$.
    Then for $t = t_0$, we have
    \begin{equation*}\label{j_f3(sigma,t) bounds-more}
    j_{f,3}(\sigma, t)\\
    \leq j_{f,3}^{\text{Main}}(\sigma, t) + j_{f,3}^{\text{Error}}(\sigma, t),
    \end{equation*}
    where
    
  \begin{equation*}\label{j_f3(sigma,t) bounds-more-main}
  j_{f,3}^{\text{Main}}(\sigma, t) = \begin{cases}
       \frac{5}{2}\kappa\log N+4\kappa\log\abs{t}+4\kappa\log(k)-\frac{2}{\sigma-\beta_0}   &\text{ if }\abs{t} \geq 1\\
       \frac{5}{2}\kappa\log N+4\kappa\log(k)-\frac{2}{\sigma-\beta_0} &\text{ if }\abs{t} < 1,\\
    \end{cases}\\
  \end{equation*}
  and 
 \begin{equation*}\label{j_f3(sigma,t) bounds-more-error}
 \begin{split}
  j_{f,3}^{\text{Error}}(\sigma, t) & =  \begin{cases}
    -1.9409 + R(N) &\text{ if }\abs{t} \geq 1\\
    -2.3414 + R(N)&\text{ if }\abs{t} < 1,\\
    \end{cases}    
\end{split}
\end{equation*}
where 
\[
    R(N) = \sum_{p\mid N} 2\log p \left(\frac{1}{p^{\sigma + 1/2} - 1} + \frac{1}{\sqrt{5}(p^{\sigma_1 + 1/2} - 1)}\right).
\]
\end{lem}
\begin{proof}
Taking the negative of the real part of the logarithmic derivative of $L_{(3)}(s,f)$ in \eqref{eqn:def-power-l-fns}, we see that
\begin{align*}
    -\Re\frac{L_{(3)}'}{L_{(3)}}\left(s,f\right)&=-2\Re\frac{L'}{L}(s,f) -\Re\frac{L'}{L}(s,\sym^3f) - \Re\left(\sum_{p\mid N}\frac{2 \log p}{p^{s} \alpha_1(p)^{-1} - 1}\right).
    \end{align*}
By \eqref{Gamma-symf}, we know that 
\[
\gamma(s, \sym^3 f) = 4 (2 \pi)^{-2s - 2k + 2} \Gamma\left(s+\frac{k-1}{2}\right)\Gamma\left(s+\frac{3}{2}(k-1)\right).
\]
Using \eqref{Hadamard-Symf} for $L(s,f)$ and $L(s,\sym^3 f)$, we get 
\begin{align}
\begin{split}\label{eqn:re-log-der-L3}
    -\Re\frac{L_{(3)}'}{L_{(3)}}\left(s,f\right)={}&\frac{5}{2}\log N -4\log(2\pi)+3\Re\frac{\Gamma'}{\Gamma}\left(s+\frac{k-1}{2}\right) + \Re\frac{\Gamma'}{\Gamma}\left(s+\frac{3}{2}(k-1)\right)\\ &-2\Re\left(\sum_{\substack{0<\Re\rho<1\\ L(\rho,f)=0}}\frac{1}{s-\rho}\right)-\Re\left(\sum_{\substack{0<\Re\rho<1\\ L(\rho,\sym^3 f)=0}}\frac{1}{s-\rho}\right)\\
     & - \Re\left(\sum_{p\mid N}\frac{2 \log p}{p^{s} \alpha_1(p)^{-1} - 1}\right).
     \end{split}
\end{align}
If $p \mid N$, then  $\alpha_1(p) = \pm 1/\sqrt p$, and so
\begin{equation}\label{eqn:ramified-sum-j3}
    \Re\left(\frac{1}{p^{s}\alpha_1(p)^{-1} - 1}\right) < \abs*{\frac{1}{p^{s}\lambda_f(p)^{-1} - 1}} < \frac{1}{p^{\Re(s) + 1/2} - 1}.
\end{equation}
After applying \eqref{eqn:ramified-sum-j3}, Lemma \ref{lem:Stechkin} (a), and Lemma \ref{lem:genLem1&2MC} in \eqref{eqn:re-log-der-L3}, we arrive to the following bound
 \begin{align}
    \begin{split}\label{eq: Intermediate Bound for jf3(sigma,mt)}
    j_{f,3}(\sigma,mt)  \leq{}& \frac{5}{2}\kappa\log N - 4\kappa\log(2\pi) + R(N)+3M\left(\frac{k-1}{2},m,t\right)\\
    & + 3C\left(\frac{k-1}{2},m,t\right)+M\left(\frac{3}{2}(k-1),m,t\right)+C\left(\frac{3}{2}(k-1),m,t\right)\\
    &-2 \Bigl( F(\sigma+imt,\rho_0)-\frac{1}{\sqrt{5}}F(\sigma_1+imt,\rho_0)\Bigr).
    \end{split}
\end{align}
For the rest of the proof, we set $m=1$, and we recall that $t=t_0$. Employing  Lemma \ref{lem:Stechkin} (b) in precisely the same way we did in Subsection~\ref{sec:jf} gives us the upper bound
\begin{align}
\begin{split}\label{eqn-j3-bound-inter}
j_{f,3}(\sigma,mt)  &\leq\frac{5}{2}\kappa\log N+3M\left(\frac{k-1}{2},1,t\right)+M\left(\frac{3}{2}(k-1),1,t\right)-\frac{2}{\sigma-\beta_0}\\&\hspace{2em}-4\kappa\log(2\pi)+3C\left(\frac{k-1}{2},1,t\right)+C\left(\frac{3}{2}(k-1),1,t\right) + R(N).
\end{split}
\end{align}
The lemma concludes by applying into \eqref{eqn-j3-bound-inter} the bounds from Corollary \ref{cor: Gamma Bound with M and C Functions} for $\abs{t}\leq1$ and $\abs{t}>1$.
\end{proof}

\subsection{Bounding \texorpdfstring{$j_{f,3}(\sigma,3t)$}{jf3(sigma,3t)}}

\begin{lem}\label{lem: j_f3(sigma,3t) bound}
    Let $\rho_0 = \beta_0 + i t_0$ be a zero of $L(s, f)$ with $1/2 < \beta_0 < 1$. 
    Then for $t = t_0$, we have
    
\begin{equation*}\label{j_f3(sigma,3t) bounds-more}
    j_{f,3}(\sigma, 3t)\\
    \leq j_{f,3}^{\text{Main}}(\sigma, 3t) + j_{f,3}^{\text{Error}}(\sigma, 3t),
    \end{equation*}
    where
    
  \begin{equation*}\label{j_f3(sigma,3t) bounds-more-main}
  j_{f,3}^{\text{Main}}(\sigma, 3t) = \begin{cases}
       \frac{5}{2}\kappa\log N+4\kappa\log\abs{t}+4\kappa\log(k) &\text{ if }\abs{t} \geq 1\\
       \frac{5}{2}\kappa\log N+4\kappa\log(k)-\frac{2(\sigma - \beta_0)}{(\sigma-\beta_0)^2 + 4t^2} &\text{ if }\abs{t} < 1\\
    \end{cases}\\
  \end{equation*}
  and 
 \begin{equation*}\label{j_f3(sigma,3t) bounds-more-error}
 \begin{split}
  j_{f,3}^{\text{Error}}(\sigma, 3t) & =  \begin{cases}
    1.20626 + R(N) &\text{ if }\abs{t} \geq 1\\
    1.03 + R(N)&\text{ if }\abs{t} < 1.\\
    \end{cases}    
\end{split}
\end{equation*}
\end{lem}

\begin{proof}
   We apply \eqref{eq: Intermediate Bound for jf3(sigma,mt)} with $m=3$ to get
   \begin{align*}
    j_{f,3}(\sigma,3t)  &\leq 
    \frac{5}{2}\kappa\log N - 4\kappa\log(2\pi) + R(N)+3M\left(\frac{k-1}{2},3,t\right)\\
    & \hspace{2em}+ 3C\left(\frac{k-1}{2},3,t\right)+M\left(\frac{3}{2}(k-1),3,t\right)+C\left(\frac{3}{2}(k-1),3,t\right)\\
    &\hspace{2em}-2 \Bigl(  \Re\Bigl( \frac{1}{\sigma + 3 i t - \rho_0} \Bigr) + \Re\Bigl( \frac{1}{\sigma - 1 + 3 i t + \overline{\rho_0}} \Bigr)-\frac{1}{\sqrt{5}}F(\sigma_1+i3t,\rho_0)\Bigr).
\end{align*}
Applying Lemma~\ref{prop:stechkin with 3t} to 
    \[
        \Re\Bigl( \frac{1}{\sigma - 1 + 3 i t + \overline{\rho_0}} \Bigr) - \frac{1}{\sqrt{5}} F(\sigma_1 + 3 i t, \rho_0)
    \]
   gives
   \begin{align}
    \begin{split}\label{eq: Intermediate Bound for jf3(sigma,3t)}
    j_{f,3}(\sigma,3t)  &\leq \frac{5}{2}\kappa\log N+3M\left(\frac{k-1}{2},3,t\right)+M\left(\frac{3}{2}(k-1),3,t\right)-\frac{2(\sigma-\beta_0)}{(\sigma-\beta_0)^2+4t^2}\\&\hspace{2em}-4\kappa\log(2\pi)+3C\left(\frac{k-1}{2},3,t\right)+C\left(\frac{3}{2}(k-1),3,t\right)+2.2473 + R(N).
    \end{split}
\end{align}
Applying the bounds for $M$ and $C$ from Corollary \ref{cor: Gamma Bound with M and C Functions} ends the proof.
\end{proof}

\subsection{Bounding \texorpdfstring{$j_{f,4}(\sigma, mt)$}{jf4(sigma,mt)}}

\begin{lem}\label{lem: j_f4(sigma,mt) bound}
Let $1 < \sigma < 1.15$.  Then, we have
\begin{equation*}\label{j_f4(sigma,mt) bounds-more}
    j_{f,4}(\sigma, mt)\\
    \leq j_{f,4}^{\text{Main}}(\sigma, mt) + j_{f,4}^{\text{Error}}(\sigma, mt),
    \end{equation*}
    where
    
  \begin{equation*}\label{j_f4(sigma,mt) bounds-more-main}
  \begin{split}
  &j_{f,4}^{\text{Main}}(\sigma, mt)  = 5\kappa \log N + 5\kappa \log k + \begin{cases}
        \frac{2}{\sigma - 1}  &\text{ if }m=0\\
        \frac{2(\sigma - 1)}{(\sigma - 1)^2 + 4t^2} &\text{ if }m=2,\,\abs{t} < 1\\
            8 \kappa \log\abs{t} &\text{ if }m = 2,\,\abs{t} \geq 1\\
        \frac{2(\sigma - 1)}{(\sigma - 1)^2 + 16t^2} &\text{ if }m=4,\,\abs{t} < 1\\
        8 \kappa \log\abs{t} &\text{ if }m = 4,\,\abs{t} \geq 1,\\    
        \end{cases}
        \end{split}
        \end{equation*}
        and
        
\begin{equation*}\label{j_f4(sigma,mt) bounds-more-error}
\begin{split}
  &j_{f,4}^{\text{Error}}(\sigma, mt)  = \begin{cases}
        {-5.42237} - 2 s(N) - s_1(N) &\text{ if }m=0\\
      {-2.92759} + 2T(N,1) + R(N,1)&\text{ if }m=2,\,\abs{t} < 1\\
           {-1.63101} + 2T(N,1) + R(N,1)&\text{ if }m = 2,\,\abs{t} \geq 1\\
         {-0.0298279} + 2T(N,1) + R(N,1)&\text{ if }m=4,\,\abs{t} < 1\\
        {-0.410404} + 2T(N,1) + R(N,1)&\text{ if }m = 4,\,\abs{t} \geq 1.\\    
        \end{cases}\\
\end{split}
\end{equation*}
Here $s(N)$ and $T(N,1)$ are as defined in Lemma \ref{lem: j_f2(sigma,mt) bound},
\begin{equation*}\label{s_1(N)}
    s_1(N) = \sum_{p\mid N} 3\log\,p \left(\frac{1}{p^{\sigma + 1} - 1} - \frac{1}{\sqrt{5}(p^{\sigma_1 +1} - 1)}\right),
    \end{equation*}
and 
\begin{equation*}\label{R(N,1)}
R(N,1) = \sum_{p\mid N} 3 \log p\left(\frac{1}{p^{\sigma+1} - 1} + \frac{1}{\sqrt{5}(p^{\sigma_1 + 1} - 1)}\right).
\end{equation*}
    
\end{lem}
\begin{proof}
    By the defintion of $L_{(4)}$ in \eqref{eqn:def-power-l-fns}, we have
    \begin{equation}\label{eq: Factorization of log derivative of f^4}
    -\frac{L_{(4)}'}{L_{(4)}}\left(s,f\right) = -2 \frac{L'}{L}(s,\chi_0) - 3\frac{L'}{L}\left(s,\sym^2f\right) - \frac{L'}{L}\left(s,\sym^4f\right) - \sum_{p\mid N}\frac{3 \log p}{p^{s+1} - 1}.
    \end{equation}
    The last term in \eqref{eq: Factorization of log derivative of f^4} is obtained by observing that $\alpha_1(p)^2 = 1/p$ if $p \mid N$.  
Using \cite[Lemma 3 and Lemma 7]{McCurley}, we have
\begin{equation}\label{Dirichlet-principal,m=0,2,4}
    \begin{split}
  &\Re \left(\frac{1}{\sqrt{5}}\frac{L'}{L}(\sigma_1 + imt,\chi_0) - \frac{L'}{L}(\sigma + imt,\chi_0)\right)   \\
  < &\begin{cases}
  \frac{1}{\sigma - 1} - 0.8973 - s(N)&\text{ if }m=0\\
      \frac{\sigma -1}{(\sigma - 1)^2 + 4t^2} + 0.1565 - \frac{\kappa}{2} \log \pi + T(N,1) &\text{ if } m = 2,\,\abs{t} < 1\\
      \frac{\kappa}{2} \log \abs{t} + 0.3530 - \frac{\kappa}{2} \log \pi + T(N,1) &\text{ if }m = 2,\,\abs{t} \geq 1\\
     \frac{\sigma -1}{(\sigma - 1)^2 + 16t^2} + 0.3636 - \frac{\kappa}{2} \log \pi + T(N,1) &\text{ if } m = 4,\,\abs{t} < 1\\ 
    \frac{\kappa}{2} \log \abs{t} + 0.4080 - \frac{\kappa}{2} \log \pi + T(N,1) &\text{ if }m = 4,\,\abs{t} \geq 1.\\ 
  \end{cases}
    \end{split}
\end{equation}
Next, we  apply \eqref{Gamma-symf}, \eqref{Hadamard-Symf}, Lemma \ref{lem:Stechkin} and the bounds in \cite[Lemma 1 and Lemma 2]{McCurley} for $m=0,2,4$ to get
\begin{equation}\label{Sym^2f,m=0,2,4}
    \begin{split}
  &\Re \left(\frac{1}{\sqrt{5}}\frac{L'}{L}(\sigma_1 + imt,\sym^2f) - \frac{L'}{L}(\sigma + imt,\sym^2f)\right)   \\
&< \kappa \log N - \frac{\kappa}{2}\log \pi - \kappa \log 2\pi+ M(k-1,m,t) + C(k-1,m,t)\\
& \hspace{2em}+
\begin{cases}
 0.2469 &\text{ if }m=0\\
 0.2469 &\text{ if }m = 2,\,\abs{t} < 1\\
 0.3915 + \frac{\kappa}{2} \log\abs{t}, &\text{ if }m = 2,\,\abs{t} \geq 1\\
 0.5842&\text{ if }m = 4,\,\abs{t} < 1\\
 \frac{\kappa}{2} \log \abs{t} + 0.4266&\text{ if }m = 4,\,\abs{t} \geq 1,
\end{cases}\\
\end{split}
\end{equation}
and 
\begin{equation}\label{Sym^4f,m=0,2,4}
    \begin{split}
  &\Re \left(\frac{1}{\sqrt{5}}\frac{L'}{L}(\sigma_1 + imt,\sym^4f) - \frac{L'}{L}(\sigma + imt,\sym^4f)\right)   \\
&< 2\kappa \log N - \frac{\kappa}{2}\log \pi - 2\kappa \log 2\pi+ M(k-1,m,t) + C(k-1,m,t) \\
&\hspace{2em}+ M(2(k-1),m,t) + C(2(k-1),m,t)+
\begin{cases}
 0.2469 &\text{ if }m=0\\
 0.2469 &\text{ if }m = 2,\,\abs{t} < 1\\
 0.3915 + \frac{\kappa}{2} \log\abs{t}, &\text{ if }m = 2,\,\abs{t} \geq 1\\
 0.5842&\text{ if }m = 4,\,\abs{t} < 1\\
 \frac{\kappa}{2} \log \abs{t} + 0.4266&\text{ if }m = 4,\,\abs{t} \geq 1.
\end{cases}\\
\end{split}
\end{equation}
We now consider the last term on the right hand side of \eqref{eq: Factorization of log derivative of f^4}.  We set
\[
    G(s) = - 3\sum_{p\mid N}\frac{\log p}{p^{s+1} - 1}.
\]
We have 
\[
    G(\sigma) -\frac{1}{\sqrt{5}} G(\sigma_1) = 3\left(\frac{1}{\sqrt{5}}\sum_{p \mid N}\frac{\log p}{p^{\sigma_1 +1} - 1} - \sum_{p \mid N}\frac{\log p}{p^{\sigma +1} - 1}\right) = -s_1(N).
\]
For $m \neq 0$, we use
\begin{equation*}
    \begin{split}
        \Re\left(G(s) - \frac{1}{\sqrt{5}}G(s_1)\right) &\leq \abs{G(s)} + \frac{\abs{G(s_1)}}{\sqrt{5}}\\
        & \leq 3\left(\frac{1}{\sqrt{5}}\sum_{p \mid N}\frac{\log p}{p^{\sigma_1 +1} - 1} + \sum_{p \mid N}\frac{\log p}{p^{\sigma +1} - 1}\right)\\& = R(N,1).
    \end{split}
\end{equation*}
Thus,
\begin{equation}\label{ramified-factors-sym4}
\Re(G(s) - G(s_1)) \leq \begin{cases}
 -s_1(N) &\text{ if }m = 0\\
 R(N,1) &\text{ if }m \neq 0.\\
 \end{cases}
 \end{equation}
We now take the Ste\v{c}kin difference of \eqref{eq: Factorization of log derivative of f^4} and apply \eqref{Dirichlet-principal,m=0,2,4}, \eqref{Sym^2f,m=0,2,4},  \eqref{Sym^4f,m=0,2,4} and \eqref{ramified-factors-sym4} to get 
\begin{equation}
\begin{split}
    &j_{f,4}(\sigma,m,t) < 5\kappa \log N - 2\kappa \log \pi - 5\kappa \log 2\pi \\
    & + 4M(k-1,m,t) + 4C(k-1,m,t) + M(2k-2,m,t) + C(2k-2,m,t)\\
    &+\ell_{f,4}(\sigma,m,t),
\end{split}
\end{equation}
where
\begin{equation*}
    \begin{split}
        \ell_{f,4}(\sigma,m,t)
        & = \begin{cases}
        \frac{2}{\sigma - 1} - 0.807 - 2 s(N) - s_1(N) &\text{ if }m=0\\
        \\
        \frac{2(\sigma - 1)}{(\sigma - 1)^2 + 4t^2} + 0.667809 + 2T(N,1) + R(N,1)&\text{ if }m=2,\,\abs{t} < 1\\
        \\
            3 \kappa \log\abs{t} + 1.6392089
             + 2T(N,1) + R(N,1)&\text{ if }m = 2,\,\abs{t} \geq 1\\
            \\
        \frac{2(\sigma - 1)}{(\sigma - 1)^2 + 16t^2} + 2.431209 + 2T(N,1) + R(N,1)&\text{ if }m=4,\,\abs{t} < 1\\
        \\
        3 \kappa \log\abs{t} + 1.8896089 + 2T(N,1) + R(N,1)&\text{ if }m = 4,\,\abs{t} \geq 1.\\    
        \end{cases}
    \end{split}
\end{equation*} 

By breaking up our bound into a main and error contribution like in Lemma \ref{lem: j_f2(sigma,mt) bound} and applying the bounds for $M$ and $C$ from Corollary \ref{cor: Gamma Bound with M and C Functions}, we conclude the lemma.
   \end{proof}

\section{Proof of Zero-Free Region for \texorpdfstring{$\abs{t}\geq 1$}{t>1}}\label{sec: ZFR for t large}

    By Proposition \ref{prop:positivity}, we have
\begin{align*}
    \begin{split}\label{eqn:positivity-applied}
    0 &\leq 8 a^2 b^2 j(\sigma, 0) + (a_2 - 4) j_{f ,2}(\sigma, 0) + 3 j_{f,4}(\sigma, 0) \\
    &{} + (a_1 - 3 a_3) j_f(\sigma, t) + 3 a_3 j_{f,3}(\sigma, t) + (a_2 - 4) j_{f,2} (\sigma, 2 t) + 4 j_{f,4}(\sigma, 2t) \\
    &{} + a_3 j_{f,3}(\sigma, 3t) + j_{f,4}(\sigma, 4t),
    \end{split}
\end{align*}
where the coefficients $a,b,a_1,a_2$ and $a_3$ are given in  \eqref{eqn:coefficients-1}, \eqref{eqn:coefficients-2} and \eqref{eqn:coefficients-3}.

Let $\beta_0 + it$ be a zero of $L(s,f)$ such that $\beta_0 > 1/2$ and $\abs{t} \geq 1$.  Applying Lemmas \ref{lem: j_f0(sigma,0) bound}, \ref{lem: j_f(sigma,t) bound}, \ref{lem: j_f2(sigma,mt) bound}, \ref{lem: j_f3(sigma,t) bound}, \ref{lem: j_f3(sigma,3t) bound}, and \ref{lem: j_f4(sigma,mt) bound}, we get, for $1 < \sigma < 1.15$,
\begin{equation}\label{inequality-1-t-large}
\begin{split}
    \frac{a_1 - 3a_3}{\sigma - \beta_0}  + \frac{6a_3}{\sigma - \beta_0}
    & < \frac{8a^2b^2 + a_2 +2}{\sigma-1} + \kappa \log N\left(\frac{a_1}{2} + 2a_2 + \frac{17}{2}a_3 + 32\right) \\
    &\hspace{2em} + (\kappa \log k + \kappa \log \abs{t})(a_1 + 2a_2 + 13a_3 + 32)\\
    & \hspace{2em}- (a_2-4)s(N) - 6s(N) -3s_1(N) + 3a_3 R(N) + (a_2-4)T(N,1) \\
    & \hspace{2em}+ 8T(N,1) + 4R(N,1) + 2T(N,1) + R(N,1) + a_3 R(N)+ C(a,b),
\end{split}
\end{equation}
where
\begin{equation*}\label{C}
\begin{split}
    C(a,b) &= 8a^2b^2(-0.601655) + (a_2 - 4)(-1.66308) + 3(-5.42237)\\
    &\hspace{2em}+(a_1 - 3a_3)(-0.48973) + 3a_3(-1.9409) + (a_2 - 4)(-0.292925)\\
    &\hspace{2em} + 4(-1.63101)+ a_3(1.20626) -0.410404\\& = -105.993.
\end{split}
\end{equation*}
Moreover, since $R(N) < \sqrt{2}T(N,1)$ and $R(N,1) < 3/2 T(N,1)$, we have
\begin{equation}\label{inequality-2-t-large}
\begin{split}
    & - (a_2-4)s(N) - 6s(N) -3s_1(N) + 3a_3 R(N) + (a_2-4)T(N,1) + 8T(N,1) + 4R(N,1)\\
  & + 2T(N,1) + R(N,1) + a_3 R(N)\\
  & < -(a_2 +2)s(N) - 3s_1(N) + \left(4a_3\sqrt{2} + a_2 + \frac{27}{2}\right) T(N,1).
\end{split}
\end{equation}
Combining \eqref{inequality-1-t-large} and \eqref{inequality-2-t-large}, we get
\begin{equation}\label{inequality-3-t-large}
\begin{split}
    \frac{a_1 + 3a_3}{\sigma - \beta_0}  
  & < \frac{8a^2b^2 + a_2 + 2}{\sigma - 1} + \kappa \log (Nk\abs{t})[a_1 + 2a_2 + 13a_3 + 32]\\
  & \hspace{2em}+ C(a,b) - (a_2 +2) s(N) - 3s_1(N)\\
  & \hspace{2em}+ \left(4a_3\sqrt{2} + a_2 + \frac{27}{2}\right) T(N,1) - \left( \frac{a_1}{2} + \frac{9a_3}{2}\right) \kappa \log N.
\end{split}
  \end{equation}  
We write
\begin{equation}\label{inequality-4-t-large}
\left(4a_3\sqrt{2} + a_2 + \frac{27}{2}\right) T(N,1)- (a_2 +2) s(N) - 3s_1(N) - \left( \frac{a_1}{2} + \frac{9a_3}{2}\right) \kappa \log N
     = \sum_{p \mid N} A_p \log p,
     \end{equation}
     where
     \begin{equation}\label{eqn:ApDef}A_p = -\kappa \frac{a_1 + 9 a_3}{2} + \frac{4 a_3 \sqrt{2} + \frac{23}{2}}{p^{\sigma} - 1} + \frac{4 a_3 \sqrt{2} + 2a_2 + \frac{31}{2}}{\sqrt{5} (p^{\sigma_1} - 1)} - \frac{9}{p^{\sigma + 1} - 1} + \frac{9}{\sqrt{5}(p^{\sigma_1 + 1} - 1)}.\end{equation}
By direct computations, we see that $A_p<0$ for all $p\geq 5$ and 
\[
    \sum_{\substack{p \mid N \\ p < 5}}A_p \log\,p + C(a,b) < A_2 \log\,2 + A_3 \log\,3 -105.9932431 < 0.
\]
Combining the above inequalities with \eqref{inequality-3-t-large}, we get
\begin{equation}\label{inequality-5-t-large}
\frac{A}{\sigma - \beta_0} < \frac{B}{\sigma - 1} + C\kappa \log (Nk\abs{t}),
\end{equation}
where $A = a_1 + 3a_3$, $B = 8a^2b^2 + a_2 + 2$ and $C = a_1 + 2a_2 + 13a_3 + 32$. 

\noindent We now set 
\[
    \sigma = 1 + \frac{r}{\log (Nk\abs{t})},
\]
for a choice of $r$ to be determined later.  By \eqref{inequality-5-t-large}, we get
\[
    \beta_0 < 1 - \left(\frac{(A-B)r - C\kappa r^2}{B + C\kappa r}\right)\frac{1}{\log (Nk\abs{t})}.
\]
The function 
\[
    f(r) = \frac{(A-B)r - C\kappa r^2}{B + C\kappa r}
\]
attains its maximum value at 
\begin{equation*}\label{eqn:optimal-r}r = \frac{\sqrt{AB} - B}{C\kappa}=0.1175,\end{equation*}
giving us the inequality
\[
    \beta_0\leq 1 - \left(\frac{A + B - 2\sqrt{AB}}{C\kappa}\right)\frac{1}{\log (Nk\abs{t})} < 1 - \frac{1}{16.7053 \log (Nk\abs{t})}.
\]
We note that if $k=2$, then there exists a newform of weight $k$ with respect to $\Gamma_0(N)$ if and only if $N \geq 2$. Thus, the above choice of $r$ ensures that $1 < \sigma < 1.15$.

Thus, if $\beta_0 + it$ is a zero of $L(s,f)$ such that $\beta_0 > 1/2$ and $\abs{t} \geq 1$, then
\begin{equation*}\label{beta-bound-t-large}
\beta_0 < 1 - \frac{1}{16.7053 \log (Nk\abs{t})}.
\end{equation*}


\section{Proof of Zero-Free Region for \texorpdfstring{$\abs{t}<1$}{t<1}}\label{sec: ZFR for t small}

In this section, we extend the zero-free region obtained in the previous section to the case $\abs{t} <1$. We consider the subcases $\abs{t}\leq \frac{\gamma}{\log kN}$ and $\frac{\gamma}{\log kN}<\abs{t}< 1$ for some carefully chosen $\gamma>0$.

Let $\beta_0 + it$ be a zero of $L(s,f)$ such that $\beta_0 > 1/2$ and $\abs{t} < 1$. We write $\beta_0=1-\frac{c}{\log kN}$ for some $c>0$. To deal with the subcase $\abs{t}\leq \frac{\gamma}{\log kN}$, we require the following result.

\begin{lem}\label{lem:positivity-t-too-small}
    For $1<\sigma<\phi$, we have 
    \[
0 \leq j(\sigma, 0) + 2 j_{f}(\sigma, 0) + j_{f ,2}(\sigma, 0) .
\]
\end{lem}

\begin{proof}
    We  shall simply expand out each of the terms above as a Dirichlet series to get that 
    \begin{align*}
        &j(\sigma, 0) + 2 j_{f}(\sigma, 0) + j_{f ,2}(\sigma, 0)\\
        &=\sum_{n\geq 1}\left(1-\frac{1}{\sqrt{5}n^{\sigma_1-\sigma}}\right)\frac{\Lambda(n)}{n^{\sigma}}\left(1+2\Lambda_f(n)+\Lambda_{f,2}(n)\right).
    \end{align*}
    The lemma follows by observing that
    \[
    \left(1+2\Lambda_f(n)+\Lambda_{f,2}(n)\right)=\left(1+\Lambda_f(n)\right)^2\geq 0.
    \]
\end{proof}

By combining equations \eqref{eq: error term for j_1 in all cases}, \eqref{jf-bound-simplified-t-0-Error}, and \eqref{eqn:jf2mt-bound-error-sigma-phi}, we see that 
\begin{align}\label{eqn:negativity-error-t-too-small}
    j_1^{\text{Error}}(\sigma,0)+2j_f^{\text{Error}}(\sigma,0)+j_{f,2}^{\text{Error}}(\sigma,0)
    &= -0.601655 + 2(0.603562)-1.36737-s(N)\nonumber\\&
    <-0.761901-s(N)\nonumber\\&<0,
\end{align}
where the last inequality follows since $s(N)>0$.  Combining Lemma \ref{lem:positivity-t-too-small} and \eqref{eqn:negativity-error-t-too-small} gives
\[
0\leq j_1^{\text{Main}}(\sigma,0)+2j_f^{\text{Main}}(\sigma,0)+j_{f,2}^{\text{Main}}(\sigma,0).
\]
Now by using  \eqref{eq: main term for small j_1 in all cases}, \eqref{jf-bound-simplified-t-0-Main}, and \eqref{jf2mt-bound-main}, we get
\begin{equation}\label{eqn:inequality_for_ZFR_too_small_t}
4\frac{\sigma-\beta_0}{(\sigma-\beta_0)^2+t^2}\leq \frac{2}{\sigma-1}+2\kappa\log N+3\kappa\log k\leq \frac{2}{\sigma-1} +3\kappa\log(kN).
\end{equation}
We  recall that $\abs{t} \leq\frac{\gamma}{\log kN}$, and we set  $\sigma = 1+\frac{r_1}{\log kN}$  for some $r_1>0$ chosen such that $1<\sigma<\phi$. Using this notation in \eqref{eqn:inequality_for_ZFR_too_small_t}, we obtain
\begin{align*}
    \frac{4(r_1+c)}{(r_1+c)^{2}+\gamma^2} \leq \frac{2}{r_1}  + 3\kappa.
\end{align*}
Solving for $c>0$ gives
\begin{equation*}\label{eqn:c-t-too-small}
    c \geq \frac{-3r_1\kappa + \sqrt{4r_1^2 - 4\gamma^2 -12 r_1 \gamma^2 \kappa - 9 r_1^2\gamma^2\kappa }}{2 + 3r_1 \kappa}.
\end{equation*}
Observe that the inequality
\begin{equation*}
 \frac{-3r_1\kappa + \sqrt{4{r_1}^2 - 4\gamma^2 -12 r_1 \gamma^2 \kappa - 9 {r_1}^2\gamma^2\kappa }}{2 + 3r_1 \kappa} \geq\frac{1}{16.7053}
\end{equation*}
is satisfied if we take $\gamma=0.30992$ and $r_1= 0.675015$. Hence, if $\beta_0+it$ is a zero of $L(s,f)$ with $\beta_0>\frac12$ and $\abs{t}\leq \frac{0.30992}{\log(kN)}$, then \begin{equation*}\label{beta-bound-t-too-small}
\beta_0 < 1 - \frac{1}{16.7053 \log (Nk)}.
\end{equation*}

If $\frac{\gamma}{\log kN}<\abs{t}< 1$, we apply Proposition \ref{prop:positivity} followed by Lemmas \ref{lem: j_f0(sigma,0) bound}, \ref{lem: j_f(sigma,t) bound}, \ref{lem: j_f2(sigma,mt) bound}, \ref{lem: j_f3(sigma,t) bound}, \ref{lem: j_f3(sigma,3t) bound}, and \ref{lem: j_f4(sigma,mt) bound}, to get

\begin{equation*}\label{eqn:main-inequality-t<1}
\begin{split}
 \frac{a_1 - 3a_3}{\sigma - \beta_0} + \frac{6a_3}{\sigma - \beta_0}  
 & < \frac{8a^2b^2 + a_2 +2}{\sigma - 1} + \left(32 + \frac{a_1}{2} + 2a_2 + \frac{17}{2}a_3\right)\kappa \log N\\
  &\hspace{2em}+ (32 + a_1 + 2a_2 + 13a_3) \kappa \log k\\
 &\hspace{2em}+ \frac{(a_2+4)(\sigma-1)}{(\sigma-1)^2+4t^2}+\frac{2(\sigma-1)}{(\sigma-1)^2+16t^2}{-\frac{2a_3(\sigma - \beta_0)}{(\sigma - \beta_0)^2 + 4t^2 }}\\
 &\hspace{2em}- (a_2 +2) s(N) - 3s_1(N) + 4a_3 R(N) + (a_2 + 6) T(N,1) + 5R(N,1)\\
 &\hspace{2em} + D(a,b),
\end{split}
\end{equation*}
where
\begin{equation*}\label{D(a,b)}
\begin{split}
D(a,b) &= -8a^2b^2(0.601655) + (a_2 - 4)(-1.66308) + 3(-5.42237)\\
&\hspace{1em}+ (a_1 - 3a_3)(-0.596438) + 3a_3(-2.3414) + (a_2 - 4)(-0.700692) + 4(-2.92759)\\
&\hspace{1em}+ 1.03a_3 -0.0298279\\&=-130.9760239.
\end{split}
\end{equation*}
It follows that 
\begin{equation}\label{inequality-2-t-not-too-small}
\begin{split}
    &\frac{a_1 + 3a_3}{\sigma - \beta_0} \\
    & < \frac{8a^2b^2 + a_2 +2}{\sigma-1} + \frac{(a_2+4)(\sigma -1)}{(\sigma-1)^2 + 4t^2} + \frac{2(\sigma -1)}{(\sigma-1)^2 + 16t^2}- \frac{2a_3(\sigma - \beta_0)}{(\sigma - \beta_0)^2 + 4t^2 }  \\&\hspace{2em}+ \kappa \log kN\left(32 + a_1 + 2 a_2 + 13a_3\right) +\sum_{p \mid N} A_p \log p +D(a,b),
\end{split}
\end{equation}
where $A_p$ is given in \eqref{eqn:ApDef}. One can easily verify that $A_p\leq 0$ for all $p\geq5$ while 
\[A_2\log2+A_3\log3< {37.5815}.\]
Using this in \eqref{inequality-2-t-not-too-small} yields
\begin{equation*}\label{inequality-3-t-not-too-small}
\begin{split}
    &\frac{a_1 + 3a_3}{\sigma - \beta_0} \\
    & < \frac{8a^2b^2 + a_2 +2}{\sigma-1} +\kappa \log kN\left(32 + a_1 + 2 a_2 + 13a_3\right)\\&+ \frac{(a_2+4)(\sigma -1)}{(\sigma-1)^2 + 4t^2} + \frac{2(\sigma -1)}{(\sigma-1)^2 + 16t^2}- \frac{2a_3(\sigma - \beta_0)}{(\sigma - \beta_0)^2 + 4t^2 }   {-93.3945}.
\end{split}
\end{equation*}
We write $\sigma=1+\frac{r_2}{\log kN}$ for some $r_2>0$ chosen such that $1<\sigma<1.15$. Using this notation and observing that
$\frac{(a_2+4)(\sigma -1)}{(\sigma-1)^2 + 4t^2} - \frac{2a_3(\sigma - \beta_0)}{(\sigma - \beta_0)^2 + 4t^2 } $ is decreasing in $t$, we get
\begin{equation}\label{eqn:section7-pre-zfr}
\begin{split}
    \frac{a_1 + 3a_3}{r_2+c} &< \frac{8a^2b^2 + a_2 +2}{r_2} + \kappa \left(32 + a_1 + 2 a_2 + 13a_3\right) +\frac{(a_2+4)r_2}{r_2^2 + 4\gamma^2} + \frac{2r_2}{r_2^2 + 16\gamma^2}- \frac{2a_3(r_2+c)}{(r_2+c)^2 + 4\gamma^2 }.
\end{split}
\end{equation}
For $\delta>0$ chosen so that
\begin{equation}\label{sec7req}
    \frac{(a_2 + 4)r_2}{r_2^2 + 4\gamma^2} + \frac{2r_2}{r_2^2 + 16\gamma^2} - \frac{2a_3(r_2 + c)}{(r_2 + c)^2 + 4\gamma^2} - \delta < 0,
\end{equation}
we see that \eqref{eqn:section7-pre-zfr} yields
\begin{equation*}\label{eqn:section7-pre-zfr-1}
\begin{split}
    \frac{A}{r_2+c} &< \frac{B}{r_2} + C\kappa  +\delta,
\end{split}
\end{equation*}
where, as in Section \ref{sec: ZFR for t large}, we set $A=a_1 + 3a_3$, $B=8a^2b^2 + a_2 +2$ and $C=32 + a_1 + 2 a_2 + 13a_3$. It follows that 
\begin{equation*}
c>\frac{(A-B)r_2-C\kappa r_2^2-\delta r_2^2}{B+C\kappa r_2+\delta r_2}.
\end{equation*}
The function 
\[
    f_{\delta}(r_2) = \frac{(A-B)r_2 - C\kappa r_2^2-\delta r_2^2}{B + C\kappa r_2+\delta r_2}
\]
attains its maximum value at 
\begin{equation*}\label{eqn:delta-optimal-r}r_2 = \frac{\sqrt{AB} - B}{C\kappa+\delta},\end{equation*}
giving us the inequality
\begin{equation*}\label{eqn:delta-optimal-c}
    c > \frac{A + B - 2\sqrt{A B}}{C \kappa + \delta}.
\end{equation*}
Substituting $r_2=\frac{\sqrt{AB} - B}{C\kappa+\delta}$ and $c=\frac{A + B - 2\sqrt{A B}}{C \kappa + \delta}$  in  \eqref{sec7req} gives
\begin{equation}\label{eqn:sec7req-refined}
    \frac{(a_2 + 4)(\sqrt{AB} - B)(C\kappa+\delta)}{\left(\sqrt{AB} - B\right)^2 + 4\gamma^2(C\kappa+\delta)^2} + \frac{2(\sqrt{AB} - B)(C\kappa+\delta)}{\left(\sqrt{AB} - B\right)^2 + 16\gamma^2(C\kappa+\delta)^2} - \frac{2a_3(A-\sqrt{AB})(C\kappa+\delta)}{(A-\sqrt{AB})^2+4\gamma^2(C\kappa+\delta)^2}- \delta < 0.
    \end{equation}

Using $\gamma=0.30992$, one can verify that \eqref{eqn:sec7req-refined} is satisfied if we take $\delta=1.62622$. With this choice of $\delta$, we get $c>\frac{1}{16.9309}$.

\section*{Acknowledgements} This research was supported by the PIMS Collaborative Research Group on $L$-functions in Analytic Number Theory and initiated at the Inclusive Paths in Explicit Number Theory Summer School hosted at BIRS-UBCO. We are thankful to BIRS, UBCO and PIMS for the financial support and for providing conducive research environments. We are also grateful to the summer school organizers for facilitating this collaboration. The research of A. H. is supported by NSERC Discovery Grant RGPIN-2025-05777. The research of S. K. is supported by the NSF GRFP under Grant No.\ DGE-2241144. The research of K. S. is supported by ANRF Core Research Grant CRG/2023/001743.

\bibliographystyle{amsalpha}
\bibliography{bibliography}

\end{document}